\documentclass[11pt]{article}

\usepackage{amsmath, amssymb, amsthm} 

\usepackage{dsfont}
\usepackage[colorlinks=true, allcolors=blue,backref]{hyperref}
\usepackage[noabbrev,capitalize,nameinlink]{cleveref}

\usepackage{tikz}
\usetikzlibrary{calc}
\usetikzlibrary{math}
\usetikzlibrary{patterns}
\usetikzlibrary{decorations.pathreplacing}
\usepackage{ifthen}

\allowdisplaybreaks
\expandafter\let\expandafter\oldproof\csname\string\proof\endcsname
\let\oldendproof\endproof
\renewenvironment{proof}[1][\proofname]{%
	\oldproof[\bf #1]%
}{\oldendproof}

\allowdisplaybreaks

\theoremstyle{plain}

\newtheorem{lemma}{Lemma}[section]
\newtheorem{theorem}[lemma]{Theorem}
\newtheorem{claim}[lemma]{Claim}
\newtheorem{proposition}[lemma]{Proposition}

\newtheorem{conjecture}[lemma]{Conjecture}

\newtheorem{definition}[lemma]{Definition}

\newtheorem{problem}[lemma]{Problem}

\theoremstyle{definition}
\newtheorem{remark}[lemma]{Remark}

\RequirePackage[normalem]{ulem} 
\RequirePackage{color}\definecolor{RED}{rgb}{1,0,0}\definecolor{BLUE}{rgb}{0,0,1} 

\newcommand*{\net}{{\mathcal C}}

\newcommand*{\headset}{{\mathcal H}}

\newcommand*{\tailset}{{\mathcal T}}
\newcommand*{\up}[1]{{#1}^{\text{U}}}
\newcommand*{\midd}[1]{{#1}^{\text{M}}}
\newcommand*{\down}[1]{{#1}^{\text{D}}}
\date{}

\title{ \vspace{-0.8cm}Ramsey problems for monotone paths in graphs and hypergraphs
}

\author{
Lior Gishboliner\thanks{Department of Mathematics, ETH, Z\"urich, Switzerland. Research supported in part by SNSF grant 200021\_196965. Email: \textbf{\{lior.gishboliner, zhihan.jin, benjamin.sudakov\}@math.ethz.ch}.}
\and 
Zhihan Jin\footnotemark[1]
\and Benny Sudakov\footnotemark[1]
}

\begin{document}

    \maketitle      

    \begin{abstract}
    \noindent
    The study of ordered Ramsey numbers of monotone paths for graphs and hypergraphs has a long history, going back to the celebrated work 
    by Erd\H{o}s and Szekeres in the early days of Ramsey theory. In this paper we obtain several results in this area, establishing two conjectures of Mubayi and Suk and improving
    bounds due to Balko, Cibulka, Kr\'al and Kyn\v{c}l. We also obtain a color-monotone version of the well-known Canonical Ramsey Theorem of Erd\H{o}s and Rado, which could be of independent interest. 
    \end{abstract}

    \section{Introduction}
    An {\em ordered $k$-uniform hypergraph} $H$ is a hypergraph whose vertices have a fixed ordering.
    The {\em ordered Ramsey number} $R_{<}(G,H)$ of ordered $k$-uniform hypergraphs $G$ and $H$ is the minimum $N$ such that every red/blue edge-coloring of the complete $k$-uniform hypergraph 
    on $\{1, \dots, N\}$ has a red copy of $G$ or a blue copy of $H$ whose vertices appear in the correct ordering.
    One of the most well-studied hypergraphs in this context is the
    monotone tight path $P_n^{(k)}$, which is the ordered $k$-uniform hypergraph with vertices $1,\dots,n$ and
    edges $(i,i+1,\dots,i+k-1)$ for $i = 1,\dots,n-k+1$. We also write $P_n := P_n^{(2)}$ for the monotone graph path with $n$ vertices. 
    

    The study of ordered Ramsey numbers dates back to the very beginning of Ramsey theory, as some of the most foundational theorems in the field fall into this framework. A key example is the celebrated Erd\H{o}s-Szekeres lemma~\cite{ES}, whose proof gives $R_{<}(P_s,P_n) = R_{<}(K_s,P_n) = (s-1)(n-1)+1$.
    
    A central feature of ordered Ramsey problems is that they often originate from and have implications to problems in geometry. 
    For example, the famous Erd\H{o}s-Szekeres cups-caps theorem~\cite{ES}, which states that $R_{<}(P_s^{(3)},P_n^{(3)}) = \binom{s+n-4}{s-2} + 1$,
    was used by Erd\H{o}s and Szekeres to prove the so-called Happy Ending Theorem in the same paper. This result states that every set of $\binom{2n-4}{n-2} + 1$ points in the plane in general position contains $n$ points in convex position. The happy ending theorem was later extended by several authors by replacing ``points" with ``convex bodies". For example, Pach and T\'oth~\cite{PT} showed that there is $N = N(n)$ such that every family of $N$ convex bodies in the plane in general position, with any two bodies having at most two common boundary points, contains $n$ bodies in convex position. Fox, Pach, Sudakov and Suk~\cite{FPSS} observed that this problem is related to the $3$-color Ramsey number of $P_n^{(3)}$, and used this connection to improve the best known bound for such an $N(n)$. This led to the further study of the multicolor Ramsey number of $P_n^{(k)}$, and this Ramsey number was determined exactly by Moshkovitz and Shapira~\cite{Moshkovitz_Shapira} and Milans, Stolee \nolinebreak and \nolinebreak West~\cite{MSW}.

   Another motivation for studying Ramsey numbers of monotone tight paths comes from the work of Mubayi and Suk~\cite{MS_k-uniform}, who observed that the $k$-uniform Ramsey number $R_{<}(P_s^{(k)},K_n^{(k)})$ is closely related to the multicolor $(k-1)$-uniform Ramsey number of cliques. Proving tight bounds for the latter is one of the major open problems in Ramsey theory (see~\cite{Sawin22} and its references for the latest results for the graph case, i.e. $k=3$).  

    A systematic study of ordered Ramsey numbers of graphs was initiated by Conlon, Fox, Lee and Sudakov~\cite{CFLS} and
    Balko, Cibulka, Kr\'al and Kyn\v{c}l~\cite{BCKK}. 
    One of the main problems considered in these works is the Ramsey problem for powers of paths.
    Let $P_n^t$ denote the $t$'th power of $P_n$, namely, the ordered graph with vertices $1,\dots,n$ and edge-set $\{(i,j) : i < j \leq i+t\}$. Note that for $P^1_n = P_n$, the Erd\H{o}s-Szekeres lemma gives $R_{<}(K_s,P_n^1) = (s-1)(n-1)+1$, which is linear in $n$ for fixed $s$. Mubayi and Suk~\cite{MS_tightPath} recently conjectured that $R_{<}(K_s,P_n^t)$ remains linear in $n$ for all fixed $s,t$, and proved the bound $R_{<}(K_{s},P_n^t) = O_{s,t}(n \log^{s-2}n)$. 
    Here, we confirm their conjecture, which can be seen as an extension of the Erd\H{o}s-Szekeres lemma.
    \begin{theorem}\label{thm:clique vs power-path}
        $R_{<}(K_{s+1},P_n^t) \leq (24s^3)^{st} n$.
    \end{theorem}

     Mubayi and Suk~\cite{MS_tightPath} also considered the ``diagonal case" of $P_n^t$ versus $K_n$, and proved the quasipolynomial bound $R_{<}(P_n^t,K_n) \leq 2^{O_t(\log^2 n)}$. Here we improve this to a polynomial bound.
    \begin{theorem}\label{thm:Kn Pn^t}
    	$R_{<}(P_n^t,K_n) \leq  2^{2t-1}n^{t(2t-1)}$. 
    \end{theorem}
    
    \noindent
It is easy to see that this Ramsey number is at least $R(K_{t+1},K_n) \geq \tilde{\Omega}(n^{(t+2)/2})$ (see~\cite{BK,Spencer}),
and therefore the exponent of $n$ should grow in terms of $t$.

The dependence on $t$ of the exponent of $n$ in the previous result can be further improved from quadratic to linear if instead of $K_n$ one takes the second graph to also be the path-power~$P_n^t$. 
This problem was first considered by Conlon, Fox, Lee and Sudakov~\cite{CFLS}, who conjectured that the Ramsey number~$R_{<}(P_n^t,P_n^t)$ is polynomial in $n$. Their actual question concerned graphs with bounded bandwidth. But since every $n$-vertex ordered graph with bandwidth $t$ is a subgraph of~$P_n^t$, this question reduces to the above conjecture. Balko, Cibulka, Kr\'al and Kyn\v{c}l~\cite{BCKK} proved the conjecture by showing that $R_{<}(P_n^t,P_n^t) = O_t(n^{129t})$. In the case $t=2$, Mubayi~\cite{Mubayi} improved the bound to $R_{<}(P_n^2,P_n^2) = O(n^{19.5})$. 
    Here we obtain the improved bound $R_{<}(P_n^t,P_n^t) = O_t(n^{4t-2})$. For $t=2$ this gives $R_{<}(P_n^2,P_n^2) = O(n^6)$. In fact, \Cref{thm:Kn Pn^t} gives $R_{<}(P_n^2,K_n) \leq 8n^6$.

    \begin{theorem}\label{thm:P^t}
		$R_{<}(P_n^t,P_n^t) \leq
        (400t^3)^{t^2}n^{4t-2}.
        $
	\end{theorem}
    \noindent

We now move on to $3$-uniform hypergraphs. 
Mubayi~\cite{Mubayi} showed that $R_{<}(K_4^{(3)},P_n^{(3)}) \leq O(n^{21})$ and conjectured that $R_{<}(K_s^{(3)},P_n^{(3)})$ is polynomial in $n$ for every fixed $s$. 
This conjecture was also reiterated by Mubayi and Suk~\cite{MS_survey}. Very recently, they~\cite{MS_tightPath} proved the quasipolynomial bound $R_{<}(K_s^{(3)},P_n^{(3)}) \leq 2^{O_s(\log^2 n)}$. Here we prove 
a polynomial bound, establishing Mubayi's conjecture. 
    \begin{theorem}\label{thm:3-uniform}
        For every $s \geq 3$, there is a constant $C = C(s)$ such that $R_{<}(K_s^{(3)},P_n^{(3)}) \leq C \cdot n^{C}$.
    \end{theorem}

    We will derive \Cref{thm:3-uniform} from another result which might be of independent interest. It can be viewed as an ordered version of the celebrated Canonical Ramsey Theorem of Erd\H{o}s and Rado~\cite{ER} (here ``ordered" refers to an order on the colors, as we shall see). To state our theorem, we need the following definition: 

    \begin{definition}\label{def:lexicographic}
        Let $\chi : \binom{N}{2} \rightarrow \mathbb{N}$. A set $x_1 < \dots < x_s$ is {\em lexicographic} if there are colors $c_1,\dots,c_{s-1} \in \mathbb{N}$ such that one of the following \nolinebreak holds:
        \begin{enumerate}
            \item $\chi(x_i,x_j) = c_i$ for all $1 \leq i < j \leq s$.
            \item $\chi(x_i,x_j) = c_{j-1}$ for all $1 \leq i < j \leq s$.
        \end{enumerate}
    \end{definition}

    The canonical Ramsey theorem~\cite{ER} states that for every $s$, there is $N = N(s)$ such that in every edge-coloring of $K_N$ (with any number of colors), there is a clique of size $s$ which is either monochromatic, rainbow or lexicographic. It is natural to restrict the number of colors to say $n$, and ask how large $N$ should be to guarantee a lexicographic $s$-clique. It is not hard to see that $N = 1+\sum_{i=0}^{s-2}n^i = O(n^{s-2})$ vertices are enough for $n$ colors. Indeed, the first vertex has degree at least $1+\sum_{i=0}^{s-3}n^i$ in one of the $n$ colors, and one can then apply induction in its neighbourhood. 
    The question becomes much more interesting, however, if we also insist that the colors are ordered, namely, that $c_1 \geq \dots \geq c_{s-1}$ in \Cref{def:lexicographic}. In this case we say that the lexicographic set is {\em non-increasing}. The following theorem shows that polynomially many vertices are still enough to find such a set. This can be viewed as an ordered version of the canonical Ramsey theorem.
    
    \begin{theorem}\label{thm:lexicographic}
        For every $s \geq 2$, there is a constant $C = C(s)$ such that every $\chi : \binom{N}{2} \rightarrow [n]$, $N = C n^C$, admits a lexicographic non-increasing set of size $s$. 
    \end{theorem}

\noindent
The idea of using {\em non-increasing sets} to bound 
$R_{<}(K_s^{(3)},P_n^{(3)})$ is due to Mubayi and Suk~\cite{MS_tightPath}, see \Cref{sec:3-uniform} for the details. 
    
\vspace{0.15cm}   
The rest of this paper is organized as follows. \Cref{thm:3-uniform,thm:lexicographic} are proved in \Cref{sec:3-uniform}. \Cref{thm:Kn Pn^t,thm:P^t} are proved in \Cref{sec:P_n^t}. The proof of \Cref{thm:clique vs power-path} is given in \Cref{sec:K_{s+1}}, and the last section contains some concluding remarks and open problems. 

\paragraph{Notation:} For two subsets of vertices $A,B$ in an ordered graph, we write $A < B$ if  $a < b$ for all $a \in A, b \in B$. 

\section{Proof of \Cref{thm:3-uniform,thm:lexicographic}}
\label{sec:3-uniform}

Throughout this section, we consider functions $\chi : \binom{N}{2} \rightarrow [n]$. We now define the notion of a {\em non-increasing set}, which plays a key role throughout this section. 

  \begin{definition}\label{def:non-increasing}
        Let $\chi : \binom{N}{2} \rightarrow [n]$. A triple $x<y<z$ is called {\em non-increasing} if $\chi(x,y) \geq \chi(y,z)$, and moreover $\chi(x,z) = \chi(x,y)$ or $\chi(x,z) = \nolinebreak \chi(y,z)$.
        A set $x_1 < \dots < x_s$ is called non-increasing if every triple $x_i,x_j,x_k$, $1 \leq i < j < k \leq s$, is non-increasing.
    \end{definition}

    Let $g(n,s)$ be the minimum $N$ such that in every coloring $\chi : \binom{N}{2} \rightarrow [n]$, there is a non-increasing set of size $s$. 
    The following was observed by Mubayi and Suk~\cite{MS_tightPath} and was their motivation for studying non-increasing sets. For completeness, we include the proof. 
    \begin{proposition}[\cite{MS_tightPath}]\label{prop:weakly non-increasing}
        $R_{<}(K_s^{(3)},P^{(3)}_n) \leq g(n-2,s)$.
    \end{proposition}
    \begin{proof}
        Fix a red/blue coloring of $K_N^{(3)}$, $N = g(n-2,s)$, and suppose that there is no blue monotone tight path with $n$ vertices. For each pair of vertices $x<y$, let $\chi(x,y)$ be the largest number of vertices in a blue monotone tight path ending at $x,y$. 
        So $2 \leq \chi(x,y) \leq n-1$. 
        Observe that if $x < y < z$ with $xyz$ blue, then $\chi(x,y) < \chi(y,z)$, because we can extend any longest path ending at $x,y$ with the edge $xyz$. Hence, a non-increasing set must be a red clique. 
        Observe that there are $n-2$ possible values of $\chi(x,y)$'s.
        By the definition of $g(n-2,s)$, there is a red clique of size $s$. 
    \end{proof}

    \vspace{0.1cm}

    \begin{remark}
        We note that the proof of \Cref{prop:weakly non-increasing} only uses that a non-increasing triple $x,y,z$ satisfies $\chi(x,y) \geq \chi(y,z)$, but does not use the full definition of a non-increasing triple. 
        Thus, one can replace $g(n-2,s)$ with the analogous function corresponding to this weaker notion of being non-increasing (only requiring $\chi(x,y) \geq \chi(y,z)$), and potentially get stronger bounds on $R_{<}(K_s^{(3)},P^{(3)}_n)$ via \Cref{prop:weakly non-increasing}. Still, we decided to stick to the stronger notion given in \Cref{def:non-increasing}, because we do not know of better bounds for the weaker notion of $g$-function (for general $s$), and also because the stronger notion is needed to prove \Cref{thm:lexicographic}. 
    \end{remark}

    \vspace{0.2cm}
    
    \noindent
    We will now prove the following theorem, which together with \Cref{prop:weakly non-increasing} implies \Cref{thm:3-uniform}.
    \begin{theorem}\label{thm:non-increasing set}
        For every $s \geq 2$, there is a constant $C = C(s)$ such that $g(n,s) \leq C \cdot n^C$.
    \end{theorem}


To prove \Cref{thm:non-increasing set}, we need to find a non-increasing set of size $s$. 
The proof is via induction on $s$, and to this end, it turns out to be convenient to find the following bigger structure.
For $s \geq 2, t \geq 1$, let $H_{s,t}$ be the ordered graph with vertices $x_1 < \dots < x_s = y_1 < \dots < y_t$ such that $\{x_1,\dots,x_s\}$ is non-increasing,  $(y_1,\dots,y_t)$ is a path, and $\chi(x_{s-1},x_s) \geq \chi(y_1,y_2) \geq \dots \geq \chi(y_{t-1},y_t)$.  
Let $f(n; s,t)$ be the minimum $N$ such that every $n$-coloring $\chi$ of the edges of $K_N$ admits a copy of $H_{s,t}$.
As $H_{s,1}$ is just a non-increasing $s$-set, it holds that $f(n;s,1) = g(n,s)$. 
The following gives a recursive bound on $f(n;s,t)$.

\begin{lemma}\label{lem:induction}
    For $s \geq 3, t \geq 1$, it holds that $f(n;s,t) \leq f(n;s-1,t+1)^{s+t} \cdot n^{s-1}$. 
\end{lemma}

\begin{proof}
    Put $M = f(n;s-1,t+1)$ and $N = M^{s+t}n^{s-1}$, and fix a coloring $\chi: \binom{N}{2} \rightarrow [n]$. Put $h = s+t-1 = |V(H_{s-1,t+1})|$.
    By definition, every $M$ vertices contain a copy of $H_{s-1,t+1}$. By double counting, there are at least $\frac{\binom{N}{M}}{\binom{N-h}{M-h}} = \frac{\binom{N}{h}}{\binom{M}{h}} \geq (N/M)^h$ copies of $H_{s-1,t+1}$. 
    By the pigeonhole principle, there is a choice of colors $\alpha_1,\dots,\alpha_{s-2},\beta \in [n]$, and a set $\mathcal{H}$ of at least $\frac{N^h}{M^h n^{s-1}}$ copies $(x_1,\dots,x_{s-1}=y_1,\dots,y_{t+1})$ of $H_{s-1,t+1}$, with the property that $\chi(x_i,x_{s-1}) = \alpha_i$ for every $1\leq i\leq s-2$, and $\chi(y_1,y_2) = \beta$. 
    By the definition of $H_{s-1,t+1}$, $\chi(x_{s-2},x_{s-1})\geq\chi(y_1,y_2)$ and the triple $x_i,x_{i+1},x_{s-1}$ is non-increasing for all $1 \le i \le s-3$.
    Thus, $\alpha_1 \geq \alpha_2 \geq \dots \geq \alpha_{s-2}\geq\beta$.
    By averaging, there is a choice of $x_1,\dots,x_{s-2},y_2,\dots,y_{t+1}$ and a set $X$ of size $|X| \geq \frac{N}{M^h n^{s-1}} = M$, such that $(x_1,\dots,x_{s-2},x,y_2,\dots,y_{t+1}) \in \mathcal{H}$ for all $x \in X$. By the definition of $M$, $X$ contains a copy of $H_{s-1,t+1}$, say on the vertices $u_1,\dots,u_{s-1}=v_1,\dots,v_{t+1}$. We now consider two cases:
    
    \paragraph{Case 1:} $\chi(v_1,v_2) \leq \beta$. 
    Then, we have $\alpha_{s-2} \geq \beta \geq \chi(v_1,v_2) \geq \dots \geq \chi(v_t,v_{t+1})$. 
    We claim that $x_1,\dots,x_{s-2},v_1,v_2,v_3,\dots,v_{t+1}$ form a copy of $H_{s,t}$ with $s$-clique $\{x_1,\dots,x_{s-2},v_1,v_2\}$ and path $v_2,\dots,v_{t+1}$. 
    It suffices to show that $\{x_1,\dots,x_{s-2},v_1,v_2\}$ is non-increasing. 
    For convenience, write $x_{s-1} := v_1$ and $x_s := v_2$. Let $1 \leq i < j < k \leq s$. If $(j,k) \neq (s-1,s)$, then 
    $x_i,x_j,x_k$ is non-increasing because 
    $\{x_1,\dots,x_{s-2},v_1\}$ and $\{x_1,\dots,x_{s-2},v_2\}$ are non-increasing, as $v_1,v_2 \in X$. 
    And for $(j,k) = (s-1,s)$, we have $\chi(x_i,x_{s-1}) = \chi(x_i,x_s) = \alpha_i$ and $\chi(x_{s-1},x_s) = \chi(v_1,v_2) \leq \beta \leq\alpha_{s-2} \leq \alpha_{i}$, meaning that $x_i,x_{s-1},x_s$ is non-increasing. 
    
    \paragraph{Case 2:} $\chi(v_1,v_2) \geq \beta$. 
    Then $\chi(u_{s-2},u_{s-1}) \geq \chi(v_1,v_2) \geq \beta$, where the first inequality is by the definition of $H_{s-1,t+1}$. Also, 
    $\chi(u_i,y_2) = \beta$ for all $1 \leq i \leq s-1$ because $u_1,\dots,u_{s-1} \in X$.
    We claim that $u_1,\dots,u_{s-1},y_2,y_3,\dots,y_{t+1}$ make a copy of $H_{s,t}$ with $s$-clique $\{u_1,\dots,u_{s-1},y_2\}$ and path $y_2,\dots,y_{t+1}$. 
    First, note that $\chi(u_{s-1},y_2) = \beta \geq \chi(y_2,y_3) \geq \dots \geq \chi(y_t,y_{t+1})$. 
    So it remains to check that $\{u_1,\dots,u_{s-1},y_2\}$ is non-increasing. 
    For convenience, write $u_s := y_2$. 
    Let $1 \leq i < j < k \leq s$. 
    If $k \leq s-1$, then the triple $u_i,u_j,u_k$ is non-increasing because 
    $\{u_1,\dots,u_{s-1}\}$ is non-increasing (by the definition of $H_{s-1,t+1}$). 
    And for $k=s$, we have 
    $\chi(u_i,u_s) = \chi(u_j,u_s) = \beta$ and 
    $\chi(u_i,u_j) \geq \chi(u_i,u_{s-1}) \geq \chi(u_{s-2},u_{s-1}) \geq \beta$, where the first two inequalities use that the triples $\{u_i,u_j,u_{s-1}\}$ and $\{u_i,u_{s-2},u_{s-1}\}$ are non-increasing. This completes the proof of the lemma. 
\end{proof}
\noindent
The following theorem implies \Cref{thm:non-increasing set}, as $f(n;s,1) = g(n,s)$.
\begin{theorem}
    For every $s \geq 2, t \geq 1$, there is a constant $C = C(s,t)$ such that $f(n;s,t) \leq O_{s,t}(n^C)$. 
    Moreover, one can take 
    $
    C(s,t) = (s+t)^{s-1} - 3(s+t)^{s-2} + \sum_{i=0}^{s-3}(s-1-i)(s+t)^i.
    $
\end{theorem}
\begin{proof}
    The proof is by induction on $s$, starting with the base case $s=2$.
    Observe that $H_{2,t}$ is just a monotone path of length $t$ with non-increasing $\chi$-labels. 
    By a result of Chvatal and Koml\'os~\cite{ChvatalKomlos}, in every edge-coloring $\chi$ of an ordered $K_N$, $N > \binom{p+q-2}{p-1}$, there are vertices $y_1 < \dots < y_{p+1}$ with $\chi(y_1,y_2) \geq \chi(y_2,y_3) \geq \dots \geq \chi(y_p,y_{p+1})$ or vertices $y_1 < \dots < y_{q+1}$ with $\chi(y_1,y_2) < \chi(y_2,y_3) < \dots < \chi(y_q,y_{q+1})$.
    Apply this with $p=t, q=n+1$, assuming $N > \binom{n+t-1}{t-1}$. The second outcome is impossible because there are only $n$ colors. And the first outcome gives a copy of $H_{2,t}$. Hence, we can take $C(2,t) = t-1$ for all $t \geq 1$. It is easy to see that this coincides with the choice of $C(s,t)$ in the statement of the theorem. 
	
    Now let $s \geq 3$. 
    By \cref{lem:induction} and the induction hypothesis, we have
    $$
	f(n;s,t) \leq f(n;s-1,t+1)^{s+t} \cdot n^{s-1} \leq 
 O_{s,t}\!\left( n^{(s+t) \cdot C(s-1,t+1) + s-1} \right),
    $$
    so one can take $C(s,t) = (s+t) \cdot C(s-1,t+1) + s - 1$. It is easy to check that the choice of $C(s,t)$ in the statement of the theorem also satisfies this recursion. 
\end{proof}

Next we prove \Cref{thm:lexicographic}. This theorem
follows by combining \Cref{thm:non-increasing set} with the following proposition:
\begin{proposition}\label{prop:lexicographic}
    For $s \geq 3$, every non-increasing set of size $2^{2s-3}$ contains a lexicographic set of size $s$. 
\end{proposition}

\noindent
It remains to prove \Cref{prop:lexicographic}. To this end, we need the following recursive definition.
\begin{definition}[weakly lexicographic]\label{def:weakly lexicographic}
    A set of size two is {\em weakly lexicographic}. For $s \geq 3$, a 
    set $x_1 < \dots < x_s$ is {\em weakly lexicographic} if one of the following holds:
    \begin{enumerate}
        \item There is a color $c \in [n]$ such that $\chi(x_1,x_i) = c$ for all $2 \leq i \leq s$, and $x_2,\dots,x_s$ is weakly lexicographic. 
        \item There is a color $c \in [n]$ such that $\chi(x_i,x_s) = c$ for all $1 \leq i \leq s-1$, and $x_1,\dots,x_{s-1}$ is weakly lexicographic. 
    \end{enumerate}
\end{definition}
Let us say that a set $x_1,\dots,x_s$ is {\em forward} (resp. {\em backward}) lexicographic if it satisfies Item 1 (resp. Item 2) in \Cref{def:lexicographic}. 
\begin{lemma}\label{lem:weakly lexicographic 1}
Let $s,t \geq 2$. Every weakly lexicographic set of size $s+t-2$ contains a forward lexicographic set of size $s$ or a backward lexicographic set of size $t$.  
\end{lemma}
\begin{proof}
    By induction on $s+t$. The base case is $s=2$ or $t=2$. This case is evident because every set of size $2$ is both forward and backward lexicographic. 
    Suppose now that $s,t \geq 3$. Let $x_1< \dots < x_{s+t-2}$ be a weakly lexicographic set. Suppose without loss of generality that Item 1 in \Cref{def:weakly lexicographic} holds. By induction, $x_2,\dots,x_{s+t-2}$ contains a forward lexicographic set of size $s-1$ or a backward lexicographic set of size $t$. 
    In the latter case, we are done; in the former case, by adding $x_1$ we get a forward lexicographic set of size $s$ (so again we are done).
\end{proof}
\begin{lemma}\label{lem:weakly lexicographic 2}
Every non-increasing set of size $2^{s-1}$ contains a weakly lexicographic set of size \nolinebreak $s$.
\end{lemma}
\begin{proof}
By induction on $s$. 
For $s=2$ this is clear. 
Let $s \geq 3$ and let $x_1 < \dots < x_k$, be a non-increasing set with $k = 2^{s-1}$. 
For each $2 \leq i \leq k-1$, it holds that $\chi(x_1,x_i) = \chi(x_1,x_k)$ or $\chi(x_i,x_k) = \chi(x_1,x_k)$ because the triple $x_1,x_i,x_k$ is non-increasing. 
Suppose, without loss of generality, that at least $(k-2)/2= 2^{s-2}-1$ of the $2 \leq i \leq k-1$ satisfy $\chi(x_1,x_i) = \chi(x_1,x_k) =: c$. 
Let~$I$ be the set consisting of these $2 \leq i \leq k-1$ and the element $k$. Then $\chi(x_1,x_i) = c$ for all $i \in I$, and $|I| \geq 2^{s-2}$.
By induction, $\{x_i : i \in I\}$ contains a weakly lexicographic set of size $s-1$, which together with $x_1$ forms a weakly lexicographic sets of size $s$.  
\end{proof}
\noindent
\Cref{prop:lexicographic} follows by combining \Cref{lem:weakly lexicographic 1,lem:weakly lexicographic 2} (where we apply \Cref{lem:weakly lexicographic 1} with $s=t$).

\section{Proof of \Cref{thm:Kn Pn^t,thm:P^t}}\label{sec:P_n^t}

\begin{definition}[$t$-clique chain]
    In an ordered graph, a {\em $t$-clique chain} consists of $t$-cliques (cliques of size $t$) $X_1,\dots,X_m$ such that for every $1 \leq i \leq m-1$, $|X_i \cap X_{i+1}| = 1$ and the last element of $X_i$ is the first element of $X_{i+1}$.
\end{definition}
\noindent
See \cref{fig:t-clique chain} for an example of a clique chain.
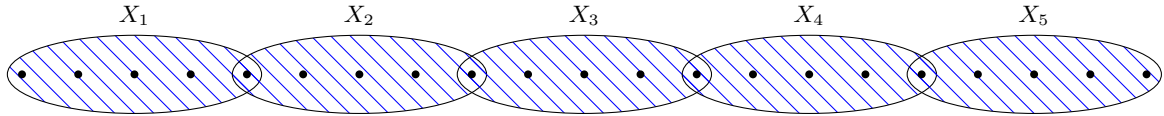
\begin{figure}
    \centering
    \begin{tikzpicture}[scale = 1.3]
        \foreach \i in {0,1,2,3,4} {
            \pgfmathtruncatemacro{\j}{\i+1};
            \coordinate (\i) at (\i*2.3, 0);
            \begin{scope}
                \clip (\i) ellipse (1.3 and 0.4);
                \foreach \x in {-10,...,1000}
                    \draw[xshift=\x*0.2 cm,blue]  (-2,2)--(2,-2);
            \end{scope}         
            \draw (\i) ellipse (1.3 and 0.4);
            \draw (\i*2.3,0.4) node[above] {\footnotesize $X_{\j}$};
        }
        \foreach \i in {-2,...,18} {
            \ifthenelse{\i=2 \OR \i=6 \OR \i=10 \OR \i=14} {
                \draw (\i*1.15/2,0) node[fill=black,circle,minimum size=3pt,inner sep=0pt] {};
            } {
                \draw (\i*1.15/2,0) node[fill=black,circle,minimum size=3pt,inner sep=0pt] {};
            }
        }
    \end{tikzpicture}
    \caption{An example of a 5-clique chain $X_1,\dots,X_5$.}
    \label{fig:t-clique chain}
\end{figure}

	\begin{lemma}\label{claim:clique_chain}
 For $n,m,t \geq 1$, the following holds:
 \begin{enumerate}
     \item Every red/blue edge-coloring of $K_N$, $N = (R(K_t,K_t) - 1) \cdot m^2 + 1$, contains a monochromatic $t$-clique chain with $m$ cliques.
     \item Every red/blue edge-coloring of $K_N$, 
     $N = (R(K_t,K_n) - 1) \cdot m + 1$, contains a blue $K_n$ or a red $t$-clique chain with $m$ cliques.
 \end{enumerate}
		
\end{lemma}
\begin{proof}
    We first prove Item 1.
    Suppose that the assertion does not hold. For each vertex $v$, let $\chi_r(v)$ (resp. $\chi_b(v)$) be the largest number of $t$-cliques in a red (resp. blue) $t$-clique chain ending at $v$. 
    Then $0 \leq \chi_r(v),\chi_b(v) \leq m-1$ for all $v$. By the pigeonhole principle, there are values $0 \leq c_r,c_b \leq m-1$ and a set $U$ with $|U| \geq N/m^2 > R(K_t,K_t)-1$, such that $\chi_r(v) = c_r$ and $\chi_b(v) = c_b$ for all $v \in U$.
    As $|U| \geq R(K_t,K_t)$, there is a monochromatic $t$-clique $x_1<\dots<x_t$ in $U$. 
    Suppose without loss of generality that this clique is red. 
    Then $\chi_r(x_t) > \chi_r(x_1)$, because any longest red $t$-clique chain ending at $x_1$ can be extended using the $t$-clique $\{x_1,\dots,x_t\}$. 
    This is a contradiction to $\chi_r(x_1) = \chi_r(x_t) = c_r$.
  
    The proof of Item 2 is similar to that of Item 1. Suppose that the statement does not hold, and let $0 \leq \chi_r(v) \leq m-1$ be defined as above. By the pigeonhole principle, there is $0 \leq c_r \leq m-1$ and a vertex-set $U$, $|U| \geq N/m$, such that $\chi_r(v) = c_r$ for all $v \in U$. 
    We have $|U| \geq R(K_t,K_n)$ by the choice of $N$.
    If $U$ contains a blue $K_n$ then we are done, and else $U$ contains a red $t$-clique $x_1 < \dots < x_t$. 
    As in the previous item, we have $\chi_r(x_t) > \chi_r(x_1)$, in contradiction to $\chi_r(x_1) = \chi_r(x_t) = c_r$. 
\end{proof}

\begin{proof}[Proof of \Cref{thm:P^t}] 
    Put $R = R(K_t,K_t)$ and $N = 2R^{2t-1}n^{4t-3} \cdot R_<(K_t,P_n^t)$, and note that $N < (400t^3)^{t^2}n^{4t-2}$, using \Cref{thm:clique vs power-path} and $R(K_t,K_t) \leq 4^t$. 
    Fix a red/blue coloring of $K_N$, and suppose by contradiction that there is no monochromatic copy of $P_n^t$. 
    For each monochromatic $t$-clique $x_1,\dots,x_t$, let $\chi(x_1,\dots,x_t)$ be the largest $\ell$ such that there is a monochromatic $P_\ell^t$ which ends at $x_1,\dots,x_t$.
    Then $t \leq \chi(x_1,\dots,x_t) \leq n-1$. A {\em good pair} is a pair of monochromatic $t$-cliques $X,Y$ in the same color, such that the last element of $X$ is the first element of $Y$, and $\chi(X) \geq \chi(Y)$. 
    \begin{claim}\label{claim:good pairs}
        There are at least $\frac{N^{2t-1}}{R^{2t-1}n^{4t-4}}$ good pairs.
    \end{claim}
    \begin{proof}
        First, we observe that every set of $(R-1) n^2+1$ vertices contains a good pair. 
        Indeed, by Item 1 of \cref{claim:clique_chain}, every $(R-1) n^2 +1$ vertices contain a monochromatic $t$-clique chain with $n$ cliques $X_1,\dots,X_{n}$. 
        There must exist an $1 \leq i \leq n-1$ such that $\chi(X_i) \geq \chi(X_{i+1})$, because $\chi(X_i) \in \{1,2,\dots,n-1\}$ for every $i$. 
        Then $(X_i,X_{i+1})$ is a good pair.
        
        Now, every set of $Rn^2$ vertices contains at least $n^2$ good pairs (by repeatedly finding a good pair and deleting one of its vertices). By double counting, there are at least 
        $$
        n^2 \cdot \frac{\binom{N}{Rn^2}}{\binom{N-2t+1}{Rn^2-2t+1}} = n^2 \cdot \frac{\binom{N}{2t-1}}{\binom{Rn^2}{2t-1}} \geq \frac{N^{2t-1}}{R^{2t-1}n^{4t-4}}
        $$
        good pairs.
    \end{proof}
	
    By \cref{claim:good pairs} and averaging, there are vertices $x_1 < \dots < x_{t-1} < z_1 < \dots < z_{t-1}$ and a set $Y'$ of vertices $x_{t-1} < y < z_1$ such that $|Y'| \geq \frac{N}{R^{2t-1}n^{4t-4}} \geq 2n \cdot R_<(K_t,P_n^t)$, and such that for every $y \in Y'$, $\{x_1,\dots,x_{t-1},y\}$ and $\{y,z_1,\dots,z_{t-1}\}$ form a good pair.
    Without loss of generality, for at least half of the vertices $y \in Y'$, the $t$-cliques $\{x_1,\dots,x_{t-1},y\},\{y,z_1,\dots,z_{t-1}\}$ are red. 
    Also, by the pigeonhole principle over the value of $\chi$, there exists a set $Y \subseteq Y'$, $|Y| \geq |Y'|/(2n) = R_<(K_t,P_n^t)$, and there exists a value $t \leq c \leq n-1$, such that $\chi(x_1,\dots,x_{t-1},y) = c$ and $\{x_1,\dots,x_{t-1},y\},\{y,z_1,\dots,z_{t-1}\}$ are red for all $y \in Y$. Note that $\chi(y,z_1,\dots,z_{t-1}) \leq c$ for all $y \in Y$, by the definition of a good pair. 
    As $|Y| = R_<(K_t,P_n^t)$ and $Y$ contains no blue $P_n^t$, it must contain a red clique $y_1<\dots<y_t$. 
    Now take a red $P^t_c$ ending at $x_1,\dots,x_{t-1},y_1$, and extend it by adding the vertices $y_2,\dots,y_t,z_1,\dots,z_{t-1}$, using that $x_1,\dots,x_{t-1},z_1,\dots,z_{t-1}$ are connected to $y_1,\dots,y_t$ in red, and that $y_1,\dots,y_t$ is a red clique. 
    It follows that $\chi(y_t,z_1,\dots,z_{t-1}) > \chi(x_1,\dots,x_{t-1},y_1) = c$, in contradiction to $\chi(y_t,z_1,\dots,z_{t-1}) \leq \nolinebreak c$.
\end{proof}

 
\begin{proof}[Proof of \Cref{thm:Kn Pn^t}]
    The proof is similar to that of \cref{thm:P^t}. 
    Put $M = R(K_t,K_n)$ and $N = (2Mn)^{2t-1}$. 
    As $R(K_t,K_n) \leq \binom{n+t-2}{t-1} \leq n^{t-1}$ (by the Erd\H{o}s-Szekeres bound~\cite{ES}), we have 
    $
    N \leq 2^{2t-1}n^{t(2t-1)}.
    $
    Fix a red/blue coloring of $K_N$,
    and suppose by contradiction that there is no red $P_n^t$ and no blue $K_n$. For each red $t$-clique $x_1,\dots,x_t$, let $\chi(x_1,\dots,x_t)$ be the largest $\ell$ such that there is a red $P_\ell^t$ that ends at $x_1,\dots,x_t$; so $t \leq \chi(x_1,\dots,x_t) \leq n-1$. 
    A {\em good pair} is a pair of red $t$-cliques $X,Y$ such that the last element of $X$ is the first element of $Y$, and $\chi(X) \geq \chi(Y)$. 
    \begin{claim}\label{claim:good pairs 2}
        There are at least $\frac{N^{2t-1}}{2^{2t-1}M^{2t-2}n^{2t-2}}$ good pairs.
    \end{claim}
    \begin{proof}
        First, we observe that every set of $(M-1) n+1$ vertices contains a good pair. Indeed, by Item 2 of \cref{claim:clique_chain}, every set of $(M-1)n+1$ vertices contains a blue $K_n$ or a red $t$-clique chain with $n$ cliques $X_1,\dots,X_{n}$. In the latter case, there must be $1 \leq i \leq n-1$ such that $\chi(X_i) \geq \chi(X_{i+1})$, which gives a good pair.
        
        Now, every set of $2 M n$ vertices contains at least $Mn$ good pairs (by repeatedly finding a good pair and deleting one of its vertices). By double counting, there are at least 
        $$
        Mn \cdot \frac{\binom{N}{2Mn}}{\binom{N-2t+1}{2Mn-2t+1}} = Mn \cdot \frac{\binom{N}{2t-1}}{\binom{2Mn}{2t-1}} \geq \frac{N^{2t-1}}{2^{2t-1}M^{2t-2}n^{2t-2}}
        $$
        good pairs.
    \end{proof}
    By \cref{claim:good pairs 2} and averaging, there are vertices $x_1 < \dots < x_{t-1} < z_1 < \dots < z_{t-1}$ and a set $Y'$ of vertices $x_{t-1} < y < z_1$ such that $|Y'| \geq \frac{N}{2^{2t-1}M^{2t-2}n^{2t-2}} \geq Mn$, and such that for all $y \in Y'$, $\{x_1,\dots,x_{t-1},y\}$ and $\{y,z_1,\dots,z_{t-1}\}$ form a good pair.
	By the pigeonhole principle over the value of $\chi$, there exists a set $Y \subseteq Y'$, $|Y| \geq |Y'|/n = M$, and a value $t \leq c \leq n-1$, such that $\chi(x_1,\dots,x_{t-1},y) = c$ for all $y \in Y$. 
    Then $\chi(y,z_1,\dots,z_{t-1}) \leq c$ for all $y \in Y$, by the definition of a good pair. 
	As $|Y| = M = R(K_t,K_n)$ and $Y$ contains no blue $K_n$, it must contain a red clique $y_1<\dots<y_t$. As in the proof of \cref{thm:P^t}, we get
	$c \geq \chi(y_t,z_1,\dots,z_{t-1}) > \chi(x_1,\dots,x_{t-1},y_1) = c$, a contradiction.
\end{proof}

\section{Proof of \Cref{thm:clique vs power-path}}\label{sec:K_{s+1}}
We begin with a brief sketch.
Our strategy for upper bounding $R(K_{s+1},P_n^t)$ is to find a certain structure which we call an {\em $s$-red-net} (see \cref{def:s-net}). We will show (see \cref{cor:s-net is enough}) that this structure implies the existence of a red $K_{s+1}$ or a blue $P_n^t$.
To find an $s$-red-net, we will first find blue cliques $V_1 < \dots < V_M$, each of large constant size, and partition each $V_i$ into $s$ consecutive equal-sized parts. 
For each $i$ and $0 \leq j \leq s-2$, we will define $\chi_j(i)$ as the largest $\ell$ such that there exists a blue $P_\ell^t$ whose last $t$ vertices belong to the first $(s-1-j)$ parts in the partition of $V_i$.
The key point is that if $i_1 < i_2$ and $\chi_j(i_1) \geq \chi_j(i_2)$, then the bipartite graph between certain parts of $V_{i_1}$ and certain parts of $V_{i_2}$ must be almost red. 
Using this, we can reduce the task of finding an $s$-red-net to the task of finding a certain structure in the functions $\chi_0,\dots,\chi_{s-2}$. We call this structure a {\em $(\chi_0,\dots,\chi_{s-2})$-forest}, see \cref{def:chi values} and \cref{lem:chi-functions}. 

We now introduce some definitions.
A {\em rooted forest} is a collection of rooted trees.
The {\em depth} of a vertex is the distance from the root (of the corresponding tree).  
A rooted forest is {\em balanced} if all leaves have the same depth. We always consider forests $F$ with $V(F) \subseteq [M]$ for some integer $M$, so that there is a natural ordering on $V(F)$. 
We say that a rooted forest is {\em well-ordered} if each vertex comes before all of its descendants, and for every two vertices $y<y'$ that are the same depth, the descendants of $y$ come before $y'$, and thus before all the descendants of $y'$. 
In particular, if $T_1,T_2$ are two components (trees) in the forest, then $T_1 < T_2$ or $T_2 < T_1$. 
Also, if $x$ is the root of a component $T$, $y_1 < \dots < y_k$ are the children of $x$, and $S_i$ is the subtree rooted at $y_i$ ($1 \leq i \leq k$), then $x < S_1 < \dots < S_k$, and each $S_i$ is well-ordered. 
Observe that if $F$ is balanced well-ordered forest of depth $d$, and $x_1,\dots,x_m$ are the roots of the components (trees) of $F$, then after deleting $x_1,\dots,x_m$ we get a balanced well-ordered forest of depth $d-1$ (the trees in this forest are the subtrees rooted at the children of $x_1,\dots,x_m$). 
We write $|F|$ for the number of vertices in a forest $F$.

We are now ready to define the notion of an $s$-red-net, which will play a key role in the proof. 

\begin{definition}\label{def:s-net}
    Consider a red/blue edge-coloring of $K_N$.
    Let $s \geq 1$. 
    An {\em $s$-red-net of order $r$} is a pair $(F,(X_v)_{v \in V(F)})$, where $F$ is a balanced well-ordered forest of depth $s-1$, and for each $v \in V(F)$, $X_v\subseteq [N]$ is a blue clique (in $K_N$) of size $r$, such that the following holds:
    \begin{enumerate}
        \item For $v,u \in V(F)$, if $v < u$ then $X_v < X_u$. 
        \item For every $v \in V(F)$ and every descendant $u$ of $v$, there is no blue $K_{t,t}$ with one part in $X_v$ and the other in $X_u$. 
    \end{enumerate}
\end{definition}

See \cref{fig:red-net-def} for an example of an $s$-red-net for $s=2,3$.
We note that since there is a total order of the vertices of $F$, there is also a corresponding total order of $(X_v)_{v\in V(F)}$.

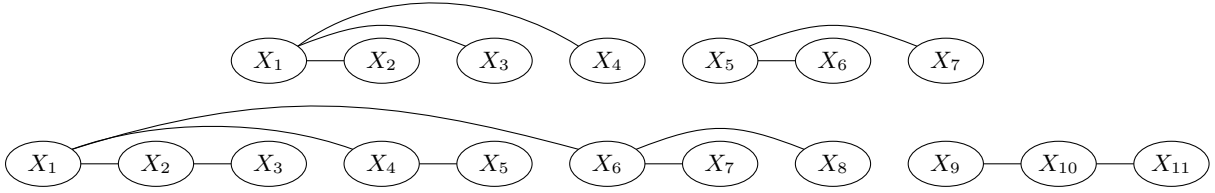
\begin{figure}
    \centering
    \begin{tikzpicture}[scale = 1]
        \coordinate (v1) at (0,0);
        \coordinate (v11) at (1.5,0);
        \coordinate (v12) at (3,0);
        \coordinate (v13) at (4.5,0);

        \draw (v1) ellipse(0.5 and 0.3);
        \draw (v11) ellipse(0.5 and 0.3);
        \draw (v12) ellipse(0.5 and 0.3);
        \draw (v13) ellipse(0.5 and 0.3);
      
        \draw (v1) node {\footnotesize  $X_1$};
        \draw (v11) node {\footnotesize $X_2$};
        \draw (v12) node {\footnotesize $X_3$};
        \draw (v13) node {\footnotesize $X_4$};

        \coordinate (v1R) at ($(v1) + ({0.5}, {0})$);
        \coordinate (v1D) at ($(v1) + ({0.5*cos(40)}, {0.3*sin(40)})$);
        \coordinate (v11L) at ($(v11) + ({-0.5}, {0})$);
        \coordinate (v11R) at ($(v11) + ({0.5}, {0})$);
        \coordinate (v12R) at ($(v12) + ({0.5}, {0})$);
        \coordinate (v13R) at ($(v13) + ({0.5}, {0})$);
        \coordinate (v12U) at ($(v12) + ({0.5*cos(140)}, {0.3*sin(140)})$);
        \coordinate (v13U) at ($(v13) + ({0.5*cos(140)}, {0.3*sin(140)})$);
        \coordinate (v13D) at ($(v13) + ({0.5*cos(40)}, {0.3*sin(40)})$);

        \draw (v1R) -- (v11L);
        \draw (v1D)  .. controls ($(v11) + ({0.5*cos(110)}, {0.3*2*sin(110)})$) and ($(v11) + ({0.5*cos(70)}, {0.3*2*sin(70)})$) .. (v12U);
        \draw (v1D)  .. controls ($(v11) + ({0.5*cos(90)}, {0.3*3.4*sin(90)})$) and ($(v12) + ({0.5*cos(90)}, {0.3*3*sin(90)})$) .. (v13U);

        \coordinate (v2) at (6,0);
        \coordinate (v21) at (7.5,0);
        \coordinate (v22) at (9,0);

        \draw (v2) ellipse(0.5 and 0.3);
        \draw (v21) ellipse(0.5 and 0.3);
        \draw (v22) ellipse(0.5 and 0.3);
      
        \draw (v2) node {\footnotesize  $X_5$};
        \draw (v21) node {\footnotesize $X_6$};
        \draw (v22) node {\footnotesize $X_7$};

        \coordinate (v2R) at ($(v2) + ({0.5}, {0})$);
        \coordinate (v2D) at ($(v2) + ({0.5*cos(40)}, {0.3*sin(40)})$);
        \coordinate (v21L) at ($(v21) + ({-0.5}, {0})$);
        \coordinate (v21R) at ($(v21) + ({0.5}, {0})$);
        \coordinate (v22R) at ($(v22) + ({0.5}, {0})$);
        \coordinate (v22U) at ($(v22) + ({0.5*cos(140)}, {0.3*sin(140)})$);

        \draw (v2R) -- (v21L);
        \draw (v2D)  .. controls ($(v21) + ({0.5*cos(110)}, {0.3*2*sin(110)})$) and ($(v21) + ({0.5*cos(70)}, {0.3*2*sin(70)})$) .. (v22U);
    \end{tikzpicture}
    \begin{tikzpicture}[scale = 1]
        \coordinate (v1) at (0,0);
        \coordinate (v11) at (1.5,0);
        \coordinate (v12) at (4.5,0);
        \coordinate (v13) at (7.5,0);
        \coordinate (v111) at (3,0);
        \coordinate (v121) at (6,0);
        \coordinate (v131) at (9,0);
        \coordinate (v132) at (10.5,0);

        \draw (v1) ellipse(0.5 and 0.3);
        \draw (v11) ellipse(0.5 and 0.3);
        \draw (v12) ellipse(0.5 and 0.3);
        \draw (v13) ellipse(0.5 and 0.3);
        \draw (v111) ellipse(0.5 and 0.3);
        \draw (v121) ellipse(0.5 and 0.3);
        \draw (v131) ellipse(0.5 and 0.3);
        \draw (v132) ellipse(0.5 and 0.3);
      
        \draw (v1) node {\footnotesize  $X_1$};
        \draw (v11) node {\footnotesize $X_{2}$};
        \draw (v12) node {\footnotesize $X_4$};
        \draw (v13) node {\footnotesize $X_6$};
        \draw (v111) node {\footnotesize $X_3$};
        \draw (v121) node {\footnotesize $X_5$};
        \draw (v131) node {\footnotesize $X_7$};
        \draw (v132) node {\footnotesize $X_8$};

        \coordinate (v1R) at ($(v1) + ({0.5}, {0})$);
        \coordinate (v1D) at ($(v1) + ({0.5*cos(40)}, {0.3*sin(40)})$);
        \coordinate (v11L) at ($(v11) + ({-0.5}, {0})$);
        \coordinate (v11R) at ($(v11) + ({0.5}, {0})$);
        \coordinate (v111L) at ($(v111) + ({-0.5}, {0})$);
        \coordinate (v12R) at ($(v12) + ({0.5}, {0})$);
        \coordinate (v121L) at ($(v121) + ({-0.5}, {0})$);
        \coordinate (v13R) at ($(v13) + ({0.5}, {0})$);
        \coordinate (v131L) at ($(v131) + ({-0.5}, {0})$);
        \coordinate (v12U) at ($(v12) + ({0.5*cos(140)}, {0.3*sin(140)})$);
        \coordinate (v13U) at ($(v13) + ({0.5*cos(140)}, {0.3*sin(140)})$);
        \coordinate (v13D) at ($(v13) + ({0.5*cos(40)}, {0.3*sin(40)})$);
        \coordinate (v132U) at ($(v132) + ({0.5*cos(140)}, {0.3*sin(140)})$);

        \draw (v1R) -- (v11L);
        \draw (v11R) -- (v111L);
        \draw (v12R) -- (v121L);
        \draw (v13R) -- (v131L);
        \draw (v1D)  .. controls ($(v11) + ({0.5*cos(90)}, {0.3*2*sin(90)})$) and ($(v111) + ({0.5*cos(90)}, {0.3*2*sin(90)})$) .. (v12U);
        \draw (v1D)  .. controls ($(v111) + ({0.5*cos(90)}, {0.3*3.4*sin(90)})$) and ($(v12) + ({0.5*cos(90)}, {0.3*3*sin(90)})$) .. (v13U);
        \draw (v13D)  .. controls ($(v131) + ({0.5*cos(110)}, {0.3*2*sin(110)})$) and ($(v131) + ({0.5*cos(70)}, {0.3*2*sin(70)})$) .. (v132U);

        \coordinate (v2) at (12,0);
        \coordinate (v21) at (13.5,0);
        \coordinate (v211) at (15,0);
        \draw (v2) node {\footnotesize  $X_9$};
        \draw (v21) node {\footnotesize $X_{10}$};
        \draw (v211) node {\footnotesize $X_{11}$};
        \draw (v2) ellipse(0.5 and 0.3);
        \draw (v21) ellipse(0.5 and 0.3);
        \draw (v211) ellipse(0.5 and 0.3);
        \coordinate (v2R) at ($(v2) + ({0.5}, {0})$);
        \coordinate (v21L) at ($(v21) + ({-0.5}, {0})$);
        \coordinate (v21R) at ($(v21) + ({0.5}, {0})$);
        \coordinate (v211L) at ($(v211) + ({-0.5}, {0})$);
        \draw (v2R) -- (v21L);
        \draw (v21R) -- (v211L);
    \end{tikzpicture}
    \caption{An example of a 2-red-net (top) and a 3-red-net (bottom). In the top example $V(F)=\{1,\dots,7\}$ and in the bottom example $V(F) = \{1,\dots,11\}$. An edge between $X_i$ and $X_j$ indicates that there is an edge between $i$ and $j$ in $F$.}
    \label{fig:red-net-def} 
\end{figure}

\subsection{Finding $s$-red-nets}
As mentioned above, we will find an $s$-red-net by finding a certain structure in a family of functions $\chi_0,\dots,\chi_{q-1} : [M] \rightarrow [n]$. We now define this structure. 
\begin{definition}\label{def:chi values}
    Let $\chi_0,\dots,\chi_{q-1} : [M] \rightarrow [n]$ be $q$ functions. A {\em $(\chi_0,\dots,\chi_{q-1})$-forest} is a well-ordered balanced forest $F$ of depth $q$ with $V(F) \subseteq [M]$, such that the following holds: For every $0 \leq d \leq q-1$ and $a \in V(F)$ at depth $d$, it holds that $\chi_{d}(a) \geq \chi_{d}(a')$ for every child $a'$ of $a$. We denote by $L(F)$ the set of leaves of $F$. 
\end{definition}
	
Note that in the case $q=1$, a $(\chi_0)$-forest simply consists of elements $x_1,\dots,x_m$ (the roots of the trees in the forest) and sets $Y_1,\dots,Y_m$ (the sets of leaves of the trees) such that $x_1 < Y_1 < \dots < x_m < Y_m$ and $\chi_0(x_i) \geq \chi_0(y)$ for all $y \in Y_i$ and $1 \leq i \leq m$. 
See \cref{fig:chi-1-2-forest} for an example when $q=2$.
\begin{figure}
    \centering
    \begin{tikzpicture}[scale = 1]
        \coordinate (v1) at (0,0);
        \coordinate (v11) at (1.5,0);
        \coordinate (v12) at (4.5,0);
        \coordinate (v13) at (7.5,0);
        \coordinate (v111) at (3,0);
        \coordinate (v121) at (6,0);
        \coordinate (v131) at (9,0);
        \coordinate (v132) at (10.5,0);
        
        \draw (v1) node[fill=black,circle,minimum size=2pt,inner sep=0pt] {};
        \draw (v11) node[fill=black,circle,minimum size=2pt,inner sep=0pt] {};
        \draw (v12) node[fill=black,circle,minimum size=2pt,inner sep=0pt] {};
        \draw (v13) node[fill=black,circle,minimum size=2pt,inner sep=0pt] {};
        \draw (v111) node[fill=black,circle,minimum size=2pt,inner sep=0pt] {};
        \draw (v121) node[fill=black,circle,minimum size=2pt,inner sep=0pt] {};
        \draw (v131) node[fill=black,circle,minimum size=2pt,inner sep=0pt] {};
        \draw (v132) node[fill=black,circle,minimum size=2pt,inner sep=0pt] {};
        
        \draw (v1) node[above] {\footnotesize $1$};
        \draw (v11) node[above] {\footnotesize $2$};
        \draw (v12) node[above] {\footnotesize $3$};
        \draw (v13) node[above] {\footnotesize $4$};
        \draw (v111) node[above] {\footnotesize $5$};
        \draw (v121) node[above] {\footnotesize $6$};
        \draw (v131) node[above] {\footnotesize $7$};
        \draw (v132) node[above] {\footnotesize $8$};
        
        \draw (v1) node[below] {\footnotesize $(7,2)$};
        \draw (v11) node[below] {\footnotesize $(2,2)$};
        \draw (v12) node[below] {\footnotesize $(4,3)$};
        \draw (v13) node[below] {\footnotesize $(5,5)$};
        \draw (v111) node[below] {\footnotesize $(3,1)$};
        \draw (v121) node[below] {\footnotesize $(2,1)$};
        \draw (v131) node[below] {\footnotesize $(6,4)$};
        \draw (v132) node[below] {\footnotesize $(5,3)$};
        
        \draw (v1) -- (v11);
        \draw (v11) -- (v111);
        \draw (v12) -- (v121);
        \draw (v13) -- (v131);
        \draw (v1)  .. controls ($(v11) + ({0.5*cos(90)}, {0.3*2.3*sin(90)})$) and ($(v111) + ({0.5*cos(90)}, {0.3*2.3*sin(90)})$) .. (v12);
        \draw (v1)  .. controls ($(v111) + ({0.5*cos(90)}, {0.3*4*sin(90)})$) and ($(v12) + ({0.5*cos(90)}, {0.3*4*sin(90)})$) .. (v13);
        \draw (v13)  .. controls ($(v131) + ({0.5*cos(120)}, {0.3*2.5*sin(120)})$) and ($(v131) + ({0.5*cos(60)}, {0.3*2.5*sin(60)})$) .. (v132);
        
        \coordinate (v2) at (12,0);
        \coordinate (v21) at (13.5,0);
        \coordinate (v211) at (15,0);
        \draw (v2) node[fill=black,circle,minimum size=2pt,inner sep=0pt] {};
        \draw (v21) node[fill=black,circle,minimum size=2pt,inner sep=0pt] {};
        \draw (v211) node[fill=black,circle,minimum size=2pt,inner sep=0pt] {};
        \draw (v2) node[above] {\footnotesize  $9$};
        \draw (v21) node[above] {\footnotesize $10$};
        \draw (v211) node[above] {\footnotesize $11$};
        \draw (v2) node[below] {\footnotesize  $(9,7)$};
        \draw (v21) node[below] {\footnotesize $(7,6)$};
        \draw (v211) node[below] {\footnotesize $(10,4)$};
        \draw (v2) -- (v21);
        \draw (v21) -- (v211);
    \end{tikzpicture}
    \caption{An example of a $(\chi_0,\chi_1)$-forest on vertex set $\{1,2,\dots,11\}$: the number above each node is its index, and the pair below node $i$ is $(\chi_0(i),\chi_1(i))$.}
    \label{fig:chi-1-2-forest} 
\end{figure}
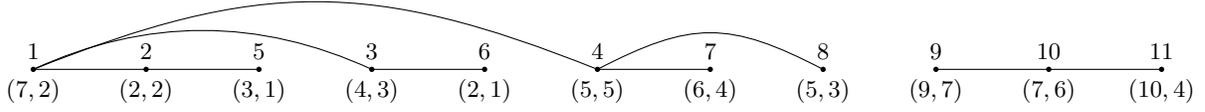
\begin{lemma}\label{lem:chi-functions}
    For every $q \geq 1$ and functions $\chi_0,\dots,\chi_{q-1} : [M] \rightarrow [n]$, 
    there is a $(\chi_0,\dots,\chi_{q-1})$-forest $F$ with 
    $|L(F)| \geq M/2^{q-1} - n$. 
\end{lemma}
\begin{proof}
    The proof is by induction on $q$. 
    First, define a sequence $x_1,x_2,\dots$ as follows. Set $x_1 = 1$, and for each $i \geq 2$, let $x_i$ be the smallest $x > x_{i-1}$ with $\chi_0(x) > \chi_0(x_{i-1})$. Let $x_1 < \dots < x_k$ be the resulting sequence. As $\chi_0(x_1) < \dots < \chi_0(x_k)$, we have $k \leq n$.
	
    First, we handle the base case $q=1$. In this case, put $Y_i=\{x_i+1,\dots,x_{i+1}-1\}$ for $1 \leq i \leq k-1$, and $Y_k=\{x_k+1,\dots,M\}$.
    By the definition of the sequence $(x_i)_i$, we have that $\chi_0(x_i) \geq \chi_0(y)$ for every $y \in Y_i$ and $1 \leq i \leq m$. 
    Now, write $I$ for the set of $i\in[k]$ with $Y_i\neq \emptyset$.
    For $i \in I$, let $T_i$ be the tree of depth $1$ with root $x_i$ and leaf-set $Y_i$. 
    Let $F$ be the forest with components $T_i$, $i \in I$. Then $F$ is well-ordered and balanced of depth $1$. 
    Also, $L(F) = |Y_1| + \dots + |Y_k| = M-k \geq M-n$. 
    So $F$ is the required $(\chi_0)$-forest.
	
	Suppose now that $q \geq 2$. 		
	Apply the induction hypothesis for $q-1$ to the functions $\chi_1,\dots,\chi_{q-1}$ and to $[M] \setminus \{x_1,\dots,x_k\}$ in place of $[M]$. This gives a $(\chi_1,\dots,\chi_{q-1})$-forest $F'$ with $V(F') \subseteq [M] \setminus \nolinebreak \{x_1,\dots,x_k\}$ and 
	$$
	|L(F')| \geq \frac{M-n}{2^{q-2}} - n \geq \frac{M}{2^{q-2}} - 2n.
	$$ 
	
	Let $S_1 < \dots < S_{\ell}$ be the components of $F'$. 
	For each $1 \leq i \leq \ell$, let $y_i$ be the root of $S_i$, and let $j_i$ be the maximum $1 \leq j \leq k$ with $x_j < y_i$ (this is well-defined because $x_1 = 1$). 
    Then $x_{j_i} < y_i < x_{j_i+1}$ if $j_i < k$, and $y_i \geq x_k$ if $j_i = k$.
    It follows that $\chi_0(x_{j_i}) \geq \chi_0(y_i)$ by the definition of the sequence $(x_j)_j$. (This means that $y_i$ is a potential child of $x_{j_i}$ in the $(\chi_0,\dots,\chi_{q-1})$-forest we are going to construct.)
    For $1 \leq j \leq k$, let $I_j = \{1 \leq i \leq \ell : j_i = j\}$, the set of potential children of $x_j$. 
    Let $J_0$ be the set of all $1 \leq j \leq k$ such that $I_j \neq \emptyset$. For each $j \in J_0$, let $Z_j$ be the minimum interval (in $[M]$) which contains the set $\{x_j\} \cup \bigcup_{i \in I_j}V(S_i)$; namely, the left endpoint of $Z_j$ is $x_j$, and the right endpoint of $Z_j$ is the rightmost element of $\bigcup_{i \in I_j}V(S_i)$. 
    See \cref{fig:proof-chi-functions} for an example.
    Recall that $L(S_i)$ denotes the set of leaves of $S_i$. 
	
    \begin{claim}
        There is $J \subseteq J_0$ such that $Z_j < Z_{j'}$ for every pair $j,j' \in J$ with $j < j'$, and 
        \begin{equation}\label{eq:bigger half}
            \sum_{j \in J}\sum_{i \in I_j}|L(S_i)| \geq \frac{1}{2}\left(
            |L(S_1)| + \dots + |L(S_{\ell})|\right) \geq 
            \frac{M}{2^{q-1}} - n. 
        \end{equation}
    \end{claim} 
    \begin{proof}
        Let $H$ be the interval graph of the intervals $Z_j$, $j \in J_0$; namely, $V(H) = J_0$ and $j,j'$ are  adjacent if $Z_j,Z_{j'}$ intersect. 
        We claim that $H$ is triangle-free. 
        Indeed, suppose that $j < j' < j''$ make a triangle in $H$. 
        Since $Z_j,Z_{j''}$ intersect, there is $v \in \bigcup_{i \in I_j}V(S_i)$ with $v > x_{j''}$. 
        Let $i \in I_j$ such that $v \in V(S_i)$. 
        We know $y_i < x_{j'}$ because $j_i = j < j'$. 
        Now, take an arbitrary $i' \in I_{j'}$ (the set $I_{j'}$ is non-empty because $j' \in J_0$).
        If $i' \le i$, then $y_{i'} \le y_i$, which means $j_{i'} \le j_i=j<j'$.
        And if $i' > i$, then $S_i < S_{i'}$. As $v \in V(S_i)$, this implies that $y_{i'} > v > x_{j''}$, so $j_{i'} \ge j'' > j'$.
        In either case, $j_{i'} \neq j'$, contradicting the fact that $i' \in I_{j'}$.
    This proves our claim that $H$ is triangle-free.
        
        Interval graphs are perfect (see~\cite[Chapter 8]{Golumbic}), so $H$ is $2$-colorable. Take $J \subseteq J_0$ to be the color class that maximizes $\sum_{j \in J}\sum_{i \in I_j}|L(S_i)|$. Then \eqref{eq:bigger half} holds. Also, for every pair $j,j' \in J$ with $j < j'$, we have $Z_{j} \cap Z_{j'} = \emptyset$ (because $J$ is independent in $H$) and hence $Z_j < Z_{j'}$ (because $Z_j,Z_{j'}$ are intervals and $x_j \in Z_j$ is to the left of $x_{j'} \in Z_{j'}$).
    \end{proof}
	
    We now complete the proof of the lemma. For each $j \in J$, form a balanced tree $T_j$ by taking~$x_j$ as the root and attaching $S_i$ as a subtree of the root for every $i \in I_j$. 
    Then $T_j$ is well-ordered and balanced of depth $q$, because the $S_i$'s are well-ordered and have depth $q-1$. 
    The children of $x_j$ are all the $y_i$ with $i \in I_j$, and we already saw that $\chi_0(x_j) \geq \chi_0(y_i)$ for each such $y_i$. 
    Also, by the claim, we have $Z_j < Z_{j'}$ for all $j,j' \in J$ with $j < j'$. 
    Hence, writing $J = \{j_1,\dots,j_{m}\}$, we have $V(T_{j_1}) < \dots < V(T_{j_m})$. 
    Let $F$ be the forest with components $T_{j_1},\dots,T_{j_m}$. 
    Then $F$ is a $(\chi_0,\dots,\chi_{q-1})$-forest; indeed, the requirement in \cref{def:chi values} holds for $d=0$ (as we just saw), and also holds for $1 \leq d \leq q-1$ because $(S_1,\dots,S_{\ell})$ is a $(\chi_1,\dots,\chi_{q-1})$-forest. 
    Also, $|L(F)| \geq M/2^{q-1} - n$ by \eqref{eq:bigger half}.
    This completes the proof.  
\end{proof}
\begin{figure}
    \centering
    \begin{tikzpicture}[scale = 1]
        \foreach \i in {1,...,15} {
            \coordinate (\i) at (\i, 0);
        }
        \draw (1) -- (15);
        
        \draw (1) node[fill=black,circle,minimum size=2pt,inner sep=0pt] {};
        \draw (5) node[fill=black,circle,minimum size=2pt,inner sep=0pt] {};
        \draw (10) node[fill=black,circle,minimum size=2pt,inner sep=0pt] {};
        \draw (13) node[fill=black,circle,minimum size=2pt,inner sep=0pt] {};
        \draw (1) node[above] {\footnotesize $x_1$};
        \draw (5) node[above] {\footnotesize $x_2$};
        \draw (10) node[above] {\footnotesize $x_3$};
        \draw (13) node[above] {\footnotesize $x_4$};
        \draw (1) node[below] {\footnotesize $1$};
        \draw (15) node[below] {\footnotesize $N$};
        
        \draw (2,0) node[fill=red,circle,minimum size=2pt,inner sep=0pt] {};
        \draw (1.9,0) node[above,red] {\footnotesize $y_1$};
        \draw[dashed,red] (2.6,0) ellipse (0.6 and 0.3);
        
        \draw (3.6,0) node[fill=red,circle,minimum size=2pt,inner sep=0pt] {};
        \draw (3.5,0) node[above,red] {\footnotesize $y_2$};
        \draw[dashed,red] (4.8,0) ellipse (1.2 and 0.5);
        
        \draw (6.4,0) node[fill=red,circle,minimum size=2pt,inner sep=0pt] {};
        \draw (6.3,0) node[above,red] {\footnotesize $y_3$};
        \draw[dashed,red] (7.2,0) ellipse (0.8 and 0.3);
        
        \draw (8.6,0) node[fill=red,circle,minimum size=2pt,inner sep=0pt] {};
        \draw (8.5,0) node[above,red] {\footnotesize $y_4$};
        \draw[dashed,red] (9.6,0) ellipse (1 and 0.4);
        
        \draw (11.6,0) node[fill=brown,circle,minimum size=2pt,inner sep=0pt] {};
        \draw (11.5,0) node[above,red] {\footnotesize $y_5$};
        \draw[dashed,red] (13.1,0) ellipse (1.5 and 0.5);

        \draw[thick,blue,<->] (1,-0.8) -- (6,-0.8);
        \draw (3.5,-0.8) node[below,blue] {\footnotesize $Z_1$};
        
        \draw[thick,blue,<->] (5,-0.6) -- (10.6,-0.6);
        \draw (7.8,-0.6) node[below,blue] {\footnotesize $Z_2$};
        
        \draw[thick,blue,<->] (10,-0.8) -- (14.6,-0.8);
        \draw (12.3,-0.8) node[below,blue] {\footnotesize $Z_3$};
    \end{tikzpicture}
    \caption{An example for the proof of \cref{lem:chi-functions} where $q=2$, $k=4$ and $\ell=5$. Every red circle stands for a component of $F'$ rooted at some $y_i$. Also, $J_0=\{1,2,3\}$, so $Z_j$ is defined for $j=1,2,3$.}
    \label{fig:proof-chi-functions} 
\end{figure}
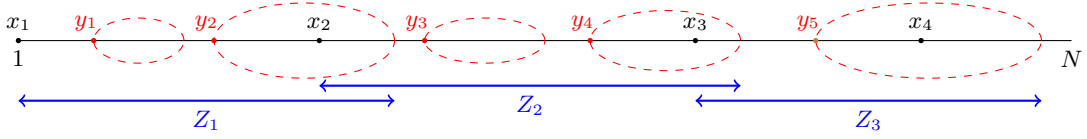

\subsection{Using $s$-red-nets}
In this section, we show in \cref{cor:s-net is enough} that a large enough $s$-red-net implies the existence of a red~$K_{s+1}$ or a blue $P_n^t$. 
This is done via a stronger statement (see \cref{lem:s-net is enough}), which has the advantage of allowing an inductive proof. 
The following definition will play an important role. 

\begin{definition}\label{def:head-tail}
    Let $\net = (F, (X_v)_v)$ be an $s$-red-net.
    Let $x$ be the root of the leftmost component (tree) of $F$ and $x'$ be the root of the rightmost component (tree) of $F$.
    The {\em head} of $\net$, denoted $\headset(\net)$, consists of the $s$ vertices of $F$ on the leftmost path from $x$ to a leaf (i.e., a path that always goes to the leftmost child). Similarly, the {\em tail} of $\net$, denoted $\tailset(\net)$, 
    consists of the $s$ vertices of $F$ on the rightmost path from $x'$ to a leaf. 
\end{definition}


As an example, if $\net$ is as depicted in the top part of \cref{fig:red-net-def}, then $\headset(\net)=\{1,2\}$ and $\tailset(\net)=\{5,7\}$, and if $\net$ is as depicted in the bottom part of \cref{fig:red-net-def}, then $\headset(\net)=\{1,2,3\}$ and $\tailset(\net)=\{9,10,11\}$. 



\begin{definition}
  For each subset $S=\{s_1,\dots,s_r\}$ of $[N]$, write $\up{S}:=\{s_1,\dots,s_{r/3}\}$, $\midd S:=\{s_{r/3+1},\dots,s_{2r/3}\}$, $\down S:=\{s_{2r/3+1},\dots,s_r\}$.\footnote{To avoid floor and ceiling signs, we will assume that $r$ is divisible by $3$.}
\end{definition}

\noindent
Our key lemma is the following.
\begin{lemma}\label{lem:s-net is enough}
    Let $s \ge 1$, $n \ge t \ge 1$ and $r \ge 3t$.
    Suppose $\net=(F,(X_v)_v)$ is an $s$-red-net of order $r$.
    Then, at least one of the following holds.
    \begin{description}
        \item[(a)] There exist $s+1$ distinct vertices $v_1,\dots,v_{s+1}\in V(F)$, along with sets 
        $A_i \subseteq X_{v_i} (\subseteq[N])$, $|A_i| \geq r/3$, such that for every $1 \le i < j \le s+1$, there is no blue $K_{t,t}$ with one part in $A_i$ and the other in $A_j$.
        \item[(b)] There exist $s$ integers $\ell_1,\dots,\ell_s \geq t$ with $\sum_{i=1}^s \ell_i \ge |F|r/3$, and there exist vertex-disjoint blue copies $P_1,\dots,P_s$ of $P_{\ell_1}^t,\dots,P_{\ell_s}^t$, respectively, such that the following holds: there exist two bijections $\sigma:[s]\to\headset(\net)$ and $\pi:[s]\to\tailset(\net)$ such that for each $i\in[s]$, the first $t$ vertices of $P_i$ lie in $\midd{X}_{\sigma(i)}$ and the last $t$ vertices of $P_i$ lie in $\midd{X}_{\pi(i)}$. 
    \end{description} 
\end{lemma}
Before proving \Cref{lem:s-net is enough}, let us sketch the proof when $s=1,2$, as these cases are easy to describe and already contain the main ideas. We will also explain how \Cref{lem:s-net is enough} is used to show that there exists a red $K_{s+1}$ or a blue $P_n^t$. 

In the case $s=1$, the forest $F$ consists of isolated vertices, say $x_1 < \dots < x_{|F|}$. 
The key observation is that if for every $1 \leq i \leq |F|-1$, there is a blue $K_{t,t}$ with parts $L_i \subseteq \down{X}_{x_i}$ and $R_i \subseteq \up{X}_{x_{i+1}}$, then we can construct a copy of $P^t_{\ell}$ (for some $\ell$) from these $K_{t,t}$'s by connecting $R_i$ and $L_{i+1}$ inside $X_{x_{i+1}}$, using that the sets $X_{x_i}$ are all blue cliques; see \cref{fig:proof-1-net}. 
Note that this $P^t_{\ell}$ contains the middle part~$\midd{X}_{x_i}$ for every $i$, so $\ell \geq 
\frac{1}{3}\sum_i |X_{x_i}| = |F|r/3$.
This case corresponds to Item (b) in the lemma. On the other hand, if, for some $1 \leq i \leq |F|-1$, there is no blue $K_{t,t}$ with one part in $\down{X}_{x_i}$ and one part in $\up{X}_{x_{i+1}}$, then Item (a) in the lemma holds with $v_1 = x_i, v_2 = x_{i+1}$. Note that $K_{t,t}$-free graphs are sparse, so here we get a bipartite graph which is very dense in red. 
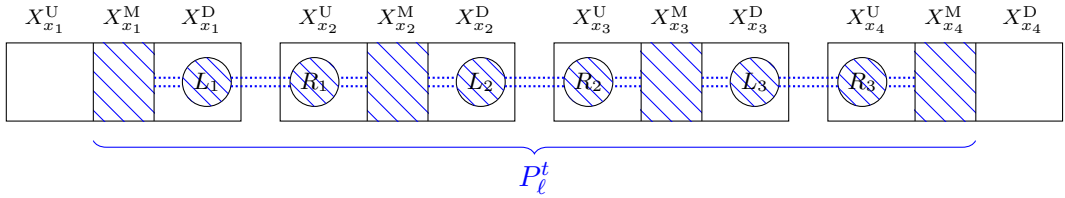
\begin{figure}
    \centering
    \begin{tikzpicture}[scale = 1.3]
        \foreach \i in {1,2,3,4} {
            \coordinate (\i) at (\i*2.8, 0);
            \draw ($(\i) + ({-1.2}, {-0.4})$) rectangle ($(\i) + ({1.2}, {0.4})$);
            \begin{scope}
                \clip ($(\i) + ({-1.2}, {-0.4})$) rectangle ($(\i) + ({1.2}, {0.4})$);
                \clip ($(\i) + ({1.2*cos(105)}, {-1})$) rectangle ($(\i) + ({1.2*cos(75)}, {1})$);
                \foreach \x in {-4,...,100}
                    \draw[xshift=\x*0.2 cm,blue]  (-2,2)--(2,-2);
            \end{scope}
            \draw (\i*2.8-0.8,0.4) node[above] {\scriptsize $\up{X}_{x_{\i}}$};
            \draw (\i*2.8,0.4) node[above] {\scriptsize $\midd{X}_{x_{\i}}$};
            \draw (\i*2.8+0.8,0.4) node[above] {\scriptsize $\down{X}_{x_{\i}}$};
            \coordinate (\i a) at ($(\i) + ({1.2*cos(75)}, {0.4})$);
            \coordinate (\i b) at ($(\i) + ({1.2*cos(105)}, {0.4})$);
            \coordinate (\i c) at ($(\i) + ({1.2*cos(255)}, {-0.4})$);
            \coordinate (\i d) at ($(\i) + ({1.2*cos(285)}, {-0.4})$);
            \coordinate (\i up) at ($(\i) + ({-0.85}, {0})$);
            \coordinate (\i down) at ($(\i) + ({0.85}, {0})$);
            \draw (\i b) -- (\i c);
            \draw (\i a) -- (\i d);
        }
        \foreach \i in {1,2,3} {
            \begin{scope}
                \clip [draw] (\i down) circle (0.25);
                \foreach \x in {-4,...,100}
                    \draw[xshift=\x*0.15 cm,blue]  (-2,2)--(2,-2);
                \draw (\i down) node {\footnotesize $L_{\i}$};
            \end{scope}
            \draw[blue,thick,densely dotted] ($(\i) + ({1.2*cos(75)}, {0.25*sin(170)})$) -- ($(\i down) + ({0.25*cos(170)}, {0.25*sin(170)})$);
            \draw[blue,thick,densely dotted] ($(\i) + ({1.2*cos(75)}, {0.25*sin(-170)})$) -- ($(\i down) + ({0.25*cos(-170)}, {0.25*sin(-170)})$);
        }
        \foreach \j in {2,3,4} {
            \pgfmathtruncatemacro{\i}{\j - 1};
            \begin{scope}
                \clip [draw] (\j up) circle (0.25);
                \foreach \x in {-4,...,100}
                    \draw[xshift=\x*0.15 cm,blue]  (-2,2)--(2,-2);
                \draw (\j up) node {\footnotesize $R_{\i}$};
            \end{scope}
            \draw[blue,thick,densely dotted] ($(\j) + ({-1.2*cos(75)}, {0.25*sin(10)})$) -- ($(\j up) + ({0.25*cos(10)}, {0.25*sin(10)})$);
            \draw[blue,thick,densely dotted] ($(\j) + ({-1.2*cos(75)}, {0.25*sin(-10)})$) -- ($(\j up) + ({0.25*cos(-10)}, {0.25*sin(-10)})$);
            \draw[blue,thick,densely dotted] ($(\i down) + ({0.25*cos(10)}, {0.25*sin(10)})$) -- ($(\j up) + ({0.25*cos(170)}, {0.25*sin(170)})$);
            \draw[blue,thick,densely dotted] ($(\i down) + ({0.25*cos(-10)}, {0.25*sin(-10)})$) -- ($(\j up) + ({0.25*cos(-170)}, {0.25*sin(-170)})$);
        }
        \draw [decorate,decoration={brace,amplitude=5pt,raise=0ex,mirror},color=blue] ($(1) + ({1.2*cos(105)}, {-0.6})$) -- ($(4) + ({1.2*cos(75)}, {-0.6})$) node[midway,yshift=-1.3em]{$P_\ell^t$};

    \end{tikzpicture}
    \caption{Proof of \Cref{lem:s-net is enough} for $s=1$: Every part with a blue shadow induces a blue clique, and every two parts connected by two blue dashed lines induce a blue complete bipartite graph.}
    \label{fig:proof-1-net} 
\end{figure}

Now let us consider the case $s=2$. In this case, the trees in $F$ have depth $1$; namely, each tree consists of a root and leaves connected to the root. 
Let $x_1 < \dots < x_m$ be the roots, and let $Y_k$ be the set of children of $x_k$. 
For each $1 \leq k \leq m$, we would like to apply the case $s=1$ to the set of vertices~$Y_k$. If Item (a) in the lemma holds, then we have vertices $v_1,v_2 \in Y_k$ and sets $A_i \subseteq X_{v_i}$ with no blue $K_{t,t}$ between $A_1,A_2$. Together with $X_{x_k}$, this gives three sets with no blue $K_{t,t}$ between any two, meaning that Item (a) in the lemma holds (for $s=2$). 
(Here we use the definition of an $s$-red-net, which implies that there is no blue $K_{t,t}$ between $X_{x_k}$ and $X_y$ for any $y \in Y_k$.)

So we may assume that when applying \Cref{lem:s-net is enough} to $Y_k$ (with $s=1$), Item (b) holds. This gives a copy $P_k$ of $P_{\ell}^t$ inside $\bigcup_{y \in Y_k} X_y$, where $\ell \geq |Y_k|r/3$. We now take another ``trivial'' $t$'th-power of a path which just consists of all the vertices in $\midd{X}_{x_k}$;
let us denote it by $P'_k$, so $|P'_k| = |\midd{X}_{x_k}|$. 
Our goal now is to connect the ends of $P_k$ and $P'_k$ to the beginnings of $P_{k+1}$ and $P'_{k+1}$. 
This connection forms a matching, i.e., we want to connect $P_k$ to $P_{k+1}$ and $P'_k$ to $P'_{k+1}$, or $P_k$ to $P'_{k+1}$ and $P_k$ to $P'_{k+1}$; see \Cref{fig:proof-2-net}.
Doing this for every $1 \leq k \leq m-1$ gives disjoint blue copies of $P^t_{\ell_1},P^t_{\ell_2}$ (for some $\ell_1,\ell_2$) which together cover $\midd{X}_x$ for every $x \in V(F)$. 
This corresponds to Item (b) in the lemma. 
As $P^t_{\ell_1},P^t_{\ell_2}$ together cover many vertices, one of them must be long, and this will give the desired \nolinebreak blue \nolinebreak $P_n^t$. 

To achieve the aforementioned connection for a specific $1 \leq k \leq m-1$, we do as follows. 
Let $u_k$ be the rightmost vertex in $Y_k\subseteq V(F)$ and let $v_{k+1}$ be the leftmost vertex in $Y_{k+1}\subseteq V(F)$; then the last $t$ vertices of $P_k$ are in $\midd{X}_{u_k}$, and the first $t$ vertices of $P_{k+1}$ are in $\midd{X}_{v_{k+1}}$.
We now define an auxiliary bipartite graph with parts $\{x_k,u_k\}$ and $\{x_{k+1},v_{k+1}\}$, where $u \in \{x_k,u_k\}$ is connected to $v \in \{x_{k+1},v_{k+1}\}$ if there is a blue $K_{t,t}$ with one part in $\down{X}_u$ and one part in $\up{X}_v$. A perfect matching in this bipartite graph gives the desired connection (see \cref{fig:proof-2-net}). On the other hand, if there is no perfect matching, then there is an isolated vertex. So suppose, for example, that $x_{k+1}$ is adjacent to neither $x_k$ nor $u_k$. Then $\down{X}_{x_k},\down{X}_{u_k},\up{X}_{x_{k+1}}$ are three sets such that between any two there is no blue $K_{t,t}$ (recall that there is no blue $K_{t,t}$ between $X_{x_k}$ and $X_{u_k}$ by the definition of an $s$-red-net). So Item (a) in the lemma holds. 
Finally, note that a bipartite graph with no blue $K_{t,t}$ is sparse in blue, and hence very dense in red. Thus, a (large enough) tripartite graph with no blue $K_{t,t}$ between any two parts must contain a red triangle
(we will argue this in \Cref{cor:s-net is enough}). We now proceed with the full proof of \Cref{lem:s-net is enough}.
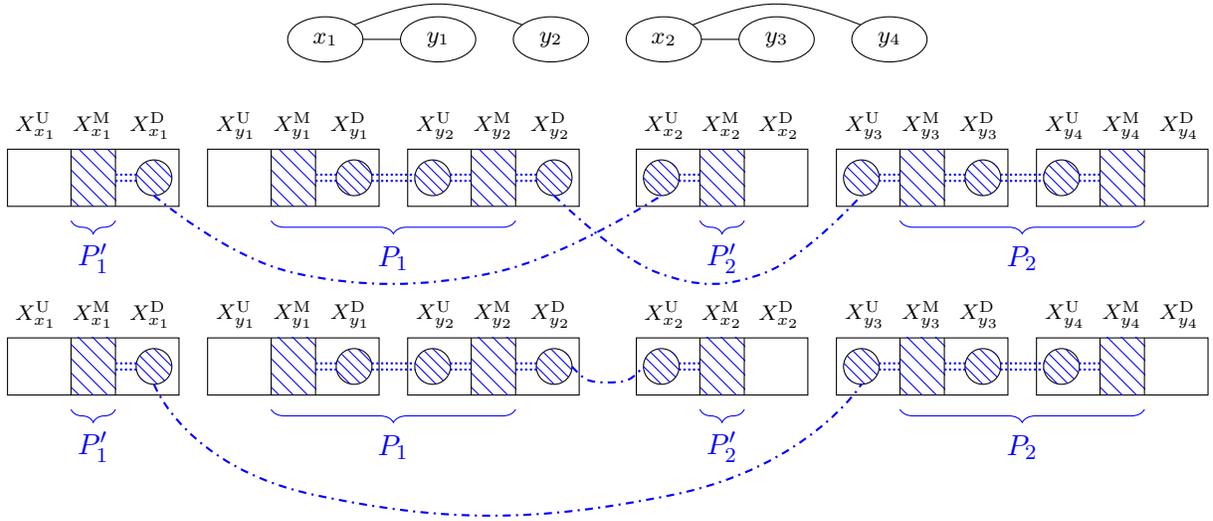
\begin{figure}
    \centering
    \begin{tikzpicture}[scale = 1]
        \coordinate (v1) at (0,0);
        \coordinate (v11) at (1.5,0);
        \coordinate (v12) at (3,0);

        \draw (v1) ellipse(0.5 and 0.3);
        \draw (v11) ellipse(0.5 and 0.3);
        \draw (v12) ellipse(0.5 and 0.3);
      
        \draw (v1) node {\footnotesize  $x_1$};
        \draw (v11) node {\footnotesize $y_1$};
        \draw (v12) node {\footnotesize $y_2$};

        \coordinate (v1R) at ($(v1) + ({0.5}, {0})$);
        \coordinate (v1D) at ($(v1) + ({0.5*cos(40)}, {0.3*sin(40)})$);
        \coordinate (v11L) at ($(v11) + ({-0.5}, {0})$);
        \coordinate (v11R) at ($(v11) + ({0.5}, {0})$);
        \coordinate (v12R) at ($(v12) + ({0.5}, {0})$);
        \coordinate (v12U) at ($(v12) + ({0.5*cos(140)}, {0.3*sin(140)})$);

        \draw (v1R) -- (v11L);
        \draw (v1D)  .. controls ($(v11) + ({0.5*cos(110)}, {0.3*2*sin(110)})$) and ($(v11) + ({0.5*cos(70)}, {0.3*2*sin(70)})$) .. (v12U);

        \coordinate (v2) at (4.5,0);
        \coordinate (v21) at (6,0);
        \coordinate (v22) at (7.5,0);

        \draw (v2) ellipse(0.5 and 0.3);
        \draw (v21) ellipse(0.5 and 0.3);
        \draw (v22) ellipse(0.5 and 0.3);
      
        \draw (v2) node {\footnotesize  $x_2$};
        \draw (v21) node {\footnotesize $y_3$};
        \draw (v22) node {\footnotesize $y_4$};

        \coordinate (v2R) at ($(v2) + ({0.5}, {0})$);
        \coordinate (v2D) at ($(v2) + ({0.5*cos(40)}, {0.3*sin(40)})$);
        \coordinate (v21L) at ($(v21) + ({-0.5}, {0})$);
        \coordinate (v21R) at ($(v21) + ({0.5}, {0})$);
        \coordinate (v22R) at ($(v22) + ({0.5}, {0})$);
        \coordinate (v22U) at ($(v22) + ({0.5*cos(140)}, {0.3*sin(140)})$);

        \draw (v2R) -- (v21L);
        \draw (v2D)  .. controls ($(v21) + ({0.5*cos(110)}, {0.3*2*sin(110)})$) and ($(v21) + ({0.5*cos(70)}, {0.3*2*sin(70)})$) .. (v22U);
    \end{tikzpicture}\\[0.5cm]
    \begin{tikzpicture}[scale = 0.95]
        \foreach \i in {1,2,3} {
            \coordinate (\i) at (\i*2.8, 0);
            \draw ($(\i) + ({-1.2}, {-0.4})$) rectangle ($(\i) + ({1.2}, {0.4})$);
            \begin{scope}
                \clip ($(\i) + ({-1.2}, {-0.4})$) rectangle ($(\i) + ({1.2}, {0.4})$);
                \clip ($(\i) + ({1.2*cos(105)}, {-1})$) rectangle ($(\i) + ({1.2*cos(75)}, {1})$);
                \foreach \x in {-4,...,100}
                    \draw[xshift=\x*0.2 cm,blue]  (-2,2)--(2,-2);
            \end{scope}
            \ifthenelse{\i=1}{
                \draw (\i*2.8-0.8,0.4) node[above] {\scriptsize $\up{X}_{x_{\i}}$};
                \draw (\i*2.8,0.4) node[above] {\scriptsize $\midd{X}_{x_{\i}}$};
                \draw (\i*2.8+0.8,0.4) node[above] {\scriptsize $\down{X}_{x_{\i}}$};
            }{
                \pgfmathtruncatemacro{\j}{\i - 1};
                \draw (\i*2.8-0.8,0.4) node[above] {\scriptsize $\up{X}_{y_{\j}}$};
                \draw (\i*2.8,0.4) node[above] {\scriptsize $\midd{X}_{y_{\j}}$};
                \draw (\i*2.8+0.8,0.4) node[above] {\scriptsize $\down{X}_{y_{\j}}$};
            }                
            \coordinate (\i a) at ($(\i) + ({1.2*cos(75)}, {0.4})$);
            \coordinate (\i b) at ($(\i) + ({1.2*cos(105)}, {0.4})$);
            \coordinate (\i c) at ($(\i) + ({1.2*cos(255)}, {-0.4})$);
            \coordinate (\i d) at ($(\i) + ({1.2*cos(285)}, {-0.4})$);
            \coordinate (\i up) at ($(\i) + ({-0.85}, {0})$);
            \coordinate (\i down) at ($(\i) + ({0.85}, {0})$);
            \draw (\i b) -- (\i c);
            \draw (\i a) -- (\i d);
        }
        \foreach \i in {1,2,3} {
            \begin{scope}
                \clip [draw] (\i down) circle (0.25);
                \foreach \x in {-4,...,100}
                    \draw[xshift=\x*0.15 cm,blue]  (-2,2)--(2,-2);
            \end{scope}
            \draw[blue,thick,densely dotted] ($(\i) + ({1.2*cos(75)}, {0.25*sin(170)})$) -- ($(\i down) + ({0.25*cos(170)}, {0.25*sin(170)})$);
            \draw[blue,thick,densely dotted] ($(\i) + ({1.2*cos(75)}, {0.25*sin(-170)})$) -- ($(\i down) + ({0.25*cos(-170)}, {0.25*sin(-170)})$);
        }
        \foreach \j in {3} {
            \begin{scope}
                \clip [draw] (\j up) circle (0.25);
                \foreach \x in {-4,...,100}
                    \draw[xshift=\x*0.15 cm,blue]  (-2,2)--(2,-2);
            \end{scope}
            \draw[blue,thick,densely dotted] ($(\j) + ({-1.2*cos(75)}, {0.25*sin(10)})$) -- ($(\j up) + ({0.25*cos(10)}, {0.25*sin(10)})$);
            \draw[blue,thick,densely dotted] ($(\j) + ({-1.2*cos(75)}, {0.25*sin(-10)})$) -- ($(\j up) + ({0.25*cos(-10)}, {0.25*sin(-10)})$);
            \pgfmathtruncatemacro{\i}{\j - 1};
            \draw[blue,thick,densely dotted] ($(\i down) + ({0.25*cos(10)}, {0.25*sin(10)})$) -- ($(\j up) + ({0.25*cos(170)}, {0.25*sin(170)})$);
            \draw[blue,thick,densely dotted] ($(\i down) + ({0.25*cos(-10)}, {0.25*sin(-10)})$) -- ($(\j up) + ({0.25*cos(-170)}, {0.25*sin(-170)})$);
        }
        \draw [decorate,decoration={brace,amplitude=5pt,raise=0ex,mirror},color=blue] ($(1) + ({1.2*cos(105)}, {-0.6})$) -- ($(1) + ({1.2*cos(75)}, {-0.6})$) node[midway,yshift=-1.3em]{$P_1'$};
        \draw [decorate,decoration={brace,amplitude=5pt,raise=0ex,mirror},color=blue] ($(2) + ({1.2*cos(105)}, {-0.6})$) -- ($(3) + ({1.2*cos(75)}, {-0.6})$) node[midway,yshift=-1.3em]{$P_1$};

        \foreach \i in {4,5,6} {
            \coordinate (\i) at (\i*2.8+0.4, 0);
            \draw ($(\i) + ({-1.2}, {-0.4})$) rectangle ($(\i) + ({1.2}, {0.4})$);
            \begin{scope}
                \clip ($(\i) + ({-1.2}, {-0.4})$) rectangle ($(\i) + ({1.2}, {0.4})$);
                \clip ($(\i) + ({1.2*cos(105)}, {-1})$) rectangle ($(\i) + ({1.2*cos(75)}, {1})$);
                \foreach \x in {-4,...,100}
                    \draw[xshift=\x*0.2 cm,blue]  (-2,2)--(2,-2);
            \end{scope}
            \ifthenelse{\i=4}{
                \draw (\i*2.8+0.4-0.8,0.4) node[above] {\scriptsize $\up{X}_{x_{2}}$};
                \draw (\i*2.8+0.4,0.4) node[above] {\scriptsize $\midd{X}_{x_{2}}$};
                \draw (\i*2.8+0.4+0.8,0.4) node[above] {\scriptsize $\down{X}_{x_{2}}$};
            }{
                \pgfmathtruncatemacro{\j}{\i - 2};
                \draw (\i*2.8+0.4-0.8,0.4) node[above] {\scriptsize $\up{X}_{y_{\j}}$};
                \draw (\i*2.8+0.4,0.4) node[above] {\scriptsize $\midd{X}_{y_{\j}}$};
                \draw (\i*2.8+0.4+0.8,0.4) node[above] {\scriptsize $\down{X}_{y_{\j}}$};
            }      
            
            \coordinate (\i a) at ($(\i) + ({1.2*cos(75)}, {0.4})$);
            \coordinate (\i b) at ($(\i) + ({1.2*cos(105)}, {0.4})$);
            \coordinate (\i c) at ($(\i) + ({1.2*cos(255)}, {-0.4})$);
            \coordinate (\i d) at ($(\i) + ({1.2*cos(285)}, {-0.4})$);
            \coordinate (\i up) at ($(\i) + ({-0.85}, {0})$);
            \coordinate (\i down) at ($(\i) + ({0.85}, {0})$);
            \draw (\i b) -- (\i c);
            \draw (\i a) -- (\i d);
        }
        \foreach \i in {5} {
            \begin{scope}
                \clip [draw] (\i down) circle (0.25);
                \foreach \x in {-4,...,200}
                    \draw[xshift=\x*0.15 cm,blue]  (-2,2)--(2,-2);
            \end{scope}
            \draw[blue,thick, densely dotted] ($(\i) + ({1.2*cos(75)}, {0.25*sin(170)})$) -- ($(\i down) + ({0.25*cos(170)}, {0.25*sin(170)})$);
            \draw[blue,thick, densely dotted] ($(\i) + ({1.2*cos(75)}, {0.25*sin(-170)})$) -- ($(\i down) + ({0.25*cos(-170)}, {0.25*sin(-170)})$);
            \pgfmathtruncatemacro{\j}{\i + 1};
            \draw[blue,thick, densely dotted] ($(\i down) + ({0.25*cos(10)}, {0.25*sin(10)})$) -- ($(\j up) + ({0.25*cos(170)}, {0.25*sin(170)})$);
            \draw[blue,thick, densely dotted] ($(\i down) + ({0.25*cos(-10)}, {0.25*sin(-10)})$) -- ($(\j up) + ({0.25*cos(-170)}, {0.25*sin(-170)})$);
        }
        \foreach \j in {4,5,6} {
            \begin{scope}
                \clip [draw] (\j up) circle (0.25);
                \foreach \x in {-4,...,200}
                    \draw[xshift=\x*0.15 cm,blue]  (-2,2)--(2,-2);
            \end{scope}
            \draw[blue,thick, densely dotted] ($(\j) + ({-1.2*cos(75)}, {0.25*sin(10)})$) -- ($(\j up) + ({0.25*cos(10)}, {0.25*sin(10)})$);
            \draw[blue,thick, densely dotted] ($(\j) + ({-1.2*cos(75)}, {0.25*sin(-10)})$) -- ($(\j up) + ({0.25*cos(-10)}, {0.25*sin(-10)})$);
        }
        \draw [decorate,decoration={brace,amplitude=5pt,raise=0ex,mirror},color=blue] ($(4) + ({1.2*cos(105)}, {-0.6})$) -- ($(4) + ({1.2*cos(75)}, {-0.6})$) node[midway,yshift=-1.3em]{$P_2'$};
        \draw [decorate,decoration={brace,amplitude=5pt,raise=0ex,mirror},color=blue] ($(5) + ({1.2*cos(105)}, {-0.6})$) -- ($(6) + ({1.2*cos(75)}, {-0.6})$) node[midway,yshift=-1.3em]{$P_2$};
        

        \begin{scope}
            \draw[blue,thick,dash dot] plot [smooth, tension=0.7] coordinates { ($(1down) + ({0.25*cos(270)},{0.25*sin(270)})$) 
($(2)+({0},{-1.33})$)  ($(3)+({0},{-1.33})$)  ($(4up)+({0.25*cos(270)},{0.25*sin(270)})$) };
        \end{scope}
        \begin{scope}
            \draw[blue,thick,dash dot] plot [smooth, tension=0.7] coordinates { ($(3down) + ({0.25*cos(270)},{0.25*sin(270)})$) 
($(4)+({-0.9},{-1.33})$)  ($(4)+({0.6},{-1.33})$)  ($(5up)+({0.25*cos(270)},{0.25*sin(270)})$) };
        \end{scope}
    \end{tikzpicture}
    \begin{tikzpicture}[scale = 0.95]
        \foreach \i in {1,2,3} {
            \coordinate (\i) at (\i*2.8, 0);
            \draw ($(\i) + ({-1.2}, {-0.4})$) rectangle ($(\i) + ({1.2}, {0.4})$);
            \begin{scope}
                \clip ($(\i) + ({-1.2}, {-0.4})$) rectangle ($(\i) + ({1.2}, {0.4})$);
                \clip ($(\i) + ({1.2*cos(105)}, {-1})$) rectangle ($(\i) + ({1.2*cos(75)}, {1})$);
                \foreach \x in {-4,...,100}
                    \draw[xshift=\x*0.2 cm,blue]  (-2,2)--(2,-2);
            \end{scope}
            \ifthenelse{\i=1}{
                \draw (\i*2.8-0.8,0.4) node[above] {\scriptsize $\up{X}_{x_{\i}}$};
                \draw (\i*2.8,0.4) node[above] {\scriptsize $\midd{X}_{x_{\i}}$};
                \draw (\i*2.8+0.8,0.4) node[above] {\scriptsize $\down{X}_{x_{\i}}$};
            }{
                \pgfmathtruncatemacro{\j}{\i - 1};
                \draw (\i*2.8-0.8,0.4) node[above] {\scriptsize $\up{X}_{y_{\j}}$};
                \draw (\i*2.8,0.4) node[above] {\scriptsize $\midd{X}_{y_{\j}}$};
                \draw (\i*2.8+0.8,0.4) node[above] {\scriptsize $\down{X}_{y_{\j}}$};
            }                
            \coordinate (\i a) at ($(\i) + ({1.2*cos(75)}, {0.4})$);
            \coordinate (\i b) at ($(\i) + ({1.2*cos(105)}, {0.4})$);
            \coordinate (\i c) at ($(\i) + ({1.2*cos(255)}, {-0.4})$);
            \coordinate (\i d) at ($(\i) + ({1.2*cos(285)}, {-0.4})$);
            \coordinate (\i up) at ($(\i) + ({-0.85}, {0})$);
            \coordinate (\i down) at ($(\i) + ({0.85}, {0})$);
            \draw (\i b) -- (\i c);
            \draw (\i a) -- (\i d);
        }
        \foreach \i in {1,2,3} {
            \begin{scope}
                \clip [draw] (\i down) circle (0.25);
                \foreach \x in {-4,...,100}
                    \draw[xshift=\x*0.15 cm,blue]  (-2,2)--(2,-2);
            \end{scope}
            \draw[blue,thick,densely dotted] ($(\i) + ({1.2*cos(75)}, {0.25*sin(170)})$) -- ($(\i down) + ({0.25*cos(170)}, {0.25*sin(170)})$);
            \draw[blue,thick,densely dotted] ($(\i) + ({1.2*cos(75)}, {0.25*sin(-170)})$) -- ($(\i down) + ({0.25*cos(-170)}, {0.25*sin(-170)})$);
        }
        \foreach \j in {3} {
            \begin{scope}
                \clip [draw] (\j up) circle (0.25);
                \foreach \x in {-4,...,100}
                    \draw[xshift=\x*0.15 cm,blue]  (-2,2)--(2,-2);
            \end{scope}
            \draw[blue,thick,densely dotted] ($(\j) + ({-1.2*cos(75)}, {0.25*sin(10)})$) -- ($(\j up) + ({0.25*cos(10)}, {0.25*sin(10)})$);
            \draw[blue,thick,densely dotted] ($(\j) + ({-1.2*cos(75)}, {0.25*sin(-10)})$) -- ($(\j up) + ({0.25*cos(-10)}, {0.25*sin(-10)})$);
            \pgfmathtruncatemacro{\i}{\j - 1};
            \draw[blue,thick,densely dotted] ($(\i down) + ({0.25*cos(10)}, {0.25*sin(10)})$) -- ($(\j up) + ({0.25*cos(170)}, {0.25*sin(170)})$);
            \draw[blue,thick,densely dotted] ($(\i down) + ({0.25*cos(-10)}, {0.25*sin(-10)})$) -- ($(\j up) + ({0.25*cos(-170)}, {0.25*sin(-170)})$);
        }
        \draw [decorate,decoration={brace,amplitude=5pt,raise=0ex,mirror},color=blue] ($(1) + ({1.2*cos(105)}, {-0.6})$) -- ($(1) + ({1.2*cos(75)}, {-0.6})$) node[midway,yshift=-1.3em]{$P_1'$};
        \draw [decorate,decoration={brace,amplitude=5pt,raise=0ex,mirror},color=blue] ($(2) + ({1.2*cos(105)}, {-0.6})$) -- ($(3) + ({1.2*cos(75)}, {-0.6})$) node[midway,yshift=-1.3em]{$P_1$};

        \foreach \i in {4,5,6} {
            \coordinate (\i) at (\i*2.8+0.4, 0);
            \draw ($(\i) + ({-1.2}, {-0.4})$) rectangle ($(\i) + ({1.2}, {0.4})$);
            \begin{scope}
                \clip ($(\i) + ({-1.2}, {-0.4})$) rectangle ($(\i) + ({1.2}, {0.4})$);
                \clip ($(\i) + ({1.2*cos(105)}, {-1})$) rectangle ($(\i) + ({1.2*cos(75)}, {1})$);
                \foreach \x in {-4,...,100}
                    \draw[xshift=\x*0.2 cm,blue]  (-2,2)--(2,-2);
            \end{scope}
            \ifthenelse{\i=4}{
                \draw (\i*2.8+0.4-0.8,0.4) node[above] {\scriptsize $\up{X}_{x_{2}}$};
                \draw (\i*2.8+0.4,0.4) node[above] {\scriptsize $\midd{X}_{x_{2}}$};
                \draw (\i*2.8+0.4+0.8,0.4) node[above] {\scriptsize $\down{X}_{x_{2}}$};
            }{
                \pgfmathtruncatemacro{\j}{\i - 2};
                \draw (\i*2.8+0.4-0.8,0.4) node[above] {\scriptsize $\up{X}_{y_{\j}}$};
                \draw (\i*2.8+0.4,0.4) node[above] {\scriptsize $\midd{X}_{y_{\j}}$};
                \draw (\i*2.8+0.4+0.8,0.4) node[above] {\scriptsize $\down{X}_{y_{\j}}$};
            }      
            
            \coordinate (\i a) at ($(\i) + ({1.2*cos(75)}, {0.4})$);
            \coordinate (\i b) at ($(\i) + ({1.2*cos(105)}, {0.4})$);
            \coordinate (\i c) at ($(\i) + ({1.2*cos(255)}, {-0.4})$);
            \coordinate (\i d) at ($(\i) + ({1.2*cos(285)}, {-0.4})$);
            \coordinate (\i up) at ($(\i) + ({-0.85}, {0})$);
            \coordinate (\i down) at ($(\i) + ({0.85}, {0})$);
            \draw (\i b) -- (\i c);
            \draw (\i a) -- (\i d);
        }
        \foreach \i in {5} {
            \begin{scope}
                \clip [draw] (\i down) circle (0.25);
                \foreach \x in {-4,...,200}
                    \draw[xshift=\x*0.15 cm,blue]  (-2,2)--(2,-2);
            \end{scope}
            \draw[blue,thick, densely dotted] ($(\i) + ({1.2*cos(75)}, {0.25*sin(170)})$) -- ($(\i down) + ({0.25*cos(170)}, {0.25*sin(170)})$);
            \draw[blue,thick, densely dotted] ($(\i) + ({1.2*cos(75)}, {0.25*sin(-170)})$) -- ($(\i down) + ({0.25*cos(-170)}, {0.25*sin(-170)})$);
            \pgfmathtruncatemacro{\j}{\i + 1};
            \draw[blue,thick, densely dotted] ($(\i down) + ({0.25*cos(10)}, {0.25*sin(10)})$) -- ($(\j up) + ({0.25*cos(170)}, {0.25*sin(170)})$);
            \draw[blue,thick, densely dotted] ($(\i down) + ({0.25*cos(-10)}, {0.25*sin(-10)})$) -- ($(\j up) + ({0.25*cos(-170)}, {0.25*sin(-170)})$);
        }
        \foreach \j in {4,5,6} {
            \begin{scope}
                \clip [draw] (\j up) circle (0.25);
                \foreach \x in {-4,...,200}
                    \draw[xshift=\x*0.15 cm,blue]  (-2,2)--(2,-2);
            \end{scope}
            \draw[blue,thick, densely dotted] ($(\j) + ({-1.2*cos(75)}, {0.25*sin(10)})$) -- ($(\j up) + ({0.25*cos(10)}, {0.25*sin(10)})$);
            \draw[blue,thick, densely dotted] ($(\j) + ({-1.2*cos(75)}, {0.25*sin(-10)})$) -- ($(\j up) + ({0.25*cos(-10)}, {0.25*sin(-10)})$);
        }
        \draw [decorate,decoration={brace,amplitude=5pt,raise=0ex,mirror},color=blue] ($(4) + ({1.2*cos(105)}, {-0.6})$) -- ($(4) + ({1.2*cos(75)}, {-0.6})$) node[midway,yshift=-1.3em]{$P_2'$};
        \draw [decorate,decoration={brace,amplitude=5pt,raise=0ex,mirror},color=blue] ($(5) + ({1.2*cos(105)}, {-0.6})$) -- ($(6) + ({1.2*cos(75)}, {-0.6})$) node[midway,yshift=-1.3em]{$P_2$};


        \begin{scope}
            \draw[blue,thick,dash dot] plot [smooth, tension=0.9] coordinates { ($(1down) + ({0.25*cos(270)},{0.25*sin(270)})$) 
($(3)+({-2.5},{-1.8})$)  ($(3)+({2.5},{-1.8})$)  ($(5up)+({0.25*cos(270)},{0.25*sin(270)})$) };
        \end{scope}
        \begin{scope}
            \draw[blue,thick,dash dot] plot [smooth, tension=0.6] coordinates { ($(3down) + ({0.25*cos(0)},{0.25*sin(0)})$)  ($(3down) + ({0.5},{-0.20})$)  ($(4up) + ({-0.5},{-0.2})$)  ($(4up)+({0.25*cos(180)},{0.25*sin(180)})$) };
        \end{scope}
    \end{tikzpicture}
    \caption{Proof of \Cref{lem:s-net is enough} for $s=2$: The top picture shows two trees (of depth one) which we want to connect. 
    The middle and bottom show the two ways of connecting $P_1,P_1'$ to $P_2,P_2'$ (the connections are the blue dotted curves).
    Every part with a blue shadow induces a blue clique and every two parts connected by a single blue dash-dotted line induce a blue complete bipartite graph.}
    \label{fig:proof-2-net} 
\end{figure}

\begin{proof}[Proof of \Cref{lem:s-net is enough}]
    The proof is by induction on $s$.
    First, for the base case $s=1$, $F$ consists of $|F|$ isolated vertices $x_1<\dots<x_{|F|}$. Note that $\mathcal{H}(\net) = \{x_1\}$ and $\mathcal{T}(\net) = \{x_{|F|}\}$.
    If, for some $1 \leq i \leq |F|-1$, there is no blue $K_{t,t}$ with one part in $\down{X}_{x_i}$ and the other in $\up{X}_{x_{i+1}}$, then (a) is satisfied with $v_1 := x_i$ and $v_2 := x_{i+1}$ with $A_1 := \down{X}_{x_i}$ and $A_2 := \up{X}_{x_{i+1}}$.
    Otherwise, for every $1 \leq i \leq |F|-1$, there exists a blue $K_{t,t}$ with one part $L_i \subset \down{X}_{x_i}$ and the other part $R_i \subset \up{X}_{x_{i+1}}$.
    As each $X_{x_i}$ forms a blue clique in $K_N$, we can connect these $|F|-1$ blue $K_{t,t}$'s by $$\midd{X}_{x_1}\to L_1\to R_1\to \midd{X}_{x_2}\to L_2\to R_2\to \midd{X}_{x_3}\to \dots\to \midd{X}_{x_{|F|-1}}\to L_{|F|-1}\to R_{|F|-1}\to \midd{X}_{x_{|F|}},$$ which forms a blue $P_{\ell}^t$, for $\ell \ge |F|r/3$, with first $t$ vertices in $X_{x_1}^M$ and last $t$ vertices in $X_{x_{|F|}}^M$.
    Thus, (b) is satisfied. 
    This completes the proof of the base case $s=1$.

    For the inductive step, let $s \ge 2$, and suppose that the lemma holds for $s-1$.
    We assume that (a) does not hold and show that (b) must hold.
    Let $x_1<\dots<x_m$ be the roots of the components (trees) in $F$.
    For each $k \in [m]$, let $F_k$ denote the forest obtained from the tree rooted at $x_k$ by deleting $x_k$ (so the components of $F_k$ are the trees rooted at the children of $x_k$). Then $F_k$ is a well-ordered balanced forest of depth $s-2$. Let $\net_k$ denote the $(s-1)$-red-net $(F_k,(X_v)_{v \in F_k})$.
    Note that $\mathcal{H}(\net) = \{x_1\} \cup \mathcal{H}(\net_1)$ and $\mathcal{T}(\net) = \{x_m\} \cup \mathcal{T}(\net_m)$. We will apply the induction hypothesis to $\net_k$ with $s-1$ in place of $s$. We now show that Item (b) of the lemma must hold:
    \begin{claim}
        For each $k \in [m]$, (b) holds for $\net_k$ in terms of $s-1$.
    \end{claim}
    \begin{proof}
        Suppose not. 
        Then by the induction hypothesis for $\net_k$ (with $s-1$ in place of $s$), (a) must hold for $\net_k$, i.e. there exist $v_1,\dots,v_s \in V(F_k)$ and $A_1,\dots,A_s$ with $A_i\subset X_{v_i}$ and $|A_i|\ge |X_{v_i}|/3$, such that for every $1 \le i < j \le s$, there is no blue $K_{t,t}$ with one part in $A_i$ and the other in $A_j$. Note that for every $1 \leq i \leq s$, there is no blue $K_{t,t}$ with one part in $X_{x_k}$ and the other in $A_i$, by the definition of an $s$-red-net (as $v_i$ is a descendant of $x_k$). 
        So the $s+1$ vertices $x_k,v_1,\dots,v_s$, along with the sets $X_{x_k},A_1,\dots,A_{s}$, satisfy (a), contradicting our assumption that (a) does not hold for~$\net$.
    \end{proof}
    By the above claim, for each $k \in [m]$, there exist integers $\ell_{k,1},\dots,\ell_{k,s-1} \geq t$ with $\sum_{i=1}^{s-1} \ell_{k,i} \ge |F_k|r/3$, there exist vertex-disjoint blue copies $P_{k,1},\dots,P_{k,s-1}$ of $P_{\ell_{k,1}}^t,\dots,P_{\ell_{k,s-1}}^t$, respectively, and there exist two bijections $\sigma_k:[s-1]\to\headset(\net_k)$ and $\pi_k:[s-1]\to\tailset(\net_k)$ such that for each $i\in[s-1]$, the first $t$ vertices of $P_{k,i}$ lie in $\midd{X}_{\sigma_k(i)}$ and the last $t$ vertices of $P_{k,i}$ lie in $\midd{X}_{\pi_k(i)}$.

    We now add an additional blue copy of $P^t_{\ell_{k,s}}$ with $\ell_{k,s}:=|\midd{X}_{x_k}|=r/3\ge t$, as follows: let $P_{k,s}$ be the copy of $P_{\ell_{k,s}}^t$ on all the vertices of $\midd{X}_{x_k}$ (recall that $X_{x_k}$ induces a blue clique by \cref{def:s-net}). 
    It will be convenient to set $\headset'(\net_k):=\{x_k\}\cup\headset(\net_k)$ and $\tailset'(\net_k):=\{x_k\}\cup\tailset(\net_k)$. 
    Note that $\headset(\net)=\headset'(\net_1)$ and $\tailset(\net) = \tailset'(\net_m)$.   
    Also, with a slight abuse of notation, we extend $\sigma_k$ and $\pi_k$ by setting $\sigma_k(s) = x_k$ and $\pi_k(s) = x_k$, to get bijections $\sigma_k : [s] \rightarrow \headset'(\net_k)$ and $\pi_k : [s] \rightarrow \tailset'(\net_k)$. 
    Note that $\sum_{i=1}^{s} \ell_{k,i} \ge (|F_k|+1)r/3$ and that $P_{k,1},\dots,P_{k,s}$ are vertex-disjoint blue copies of $P_{\ell_{k,1}}^t,\dots,P_{\ell_{k,s}}^t$, respectively, such that for each $i\in[s]$, the first $t$ vertices of $P_{k,i}$ are in $\midd{X}_{\sigma_k(i)}$ and the last $t$ vertices of $P_{k,i}$ are in $\midd{X}_{\pi_k(i)}$

    To prove that Item (b) holds, we need to ``connect" $P_{k,1},\dots,P_{k,s}$ to $P_{k+1,1},\dots,P_{k+1,s}$ for every $1 \leq k \leq m-1$. More precisely, we will prove the following by induction on $k$:

    \begin{description}
        \item[(b')]
        For every $1 \leq k \leq m$, there exist integers $\ell_1,\dots,\ell_s \geq t$ with $\sum_{i=1}^s \ell_i \ge \sum_{j=1}^k (|F_j|+1)r/3$, and there exist vertex-disjoint copies $P_1,\dots,P_s$ of $P_{\ell_1}^t,\dots,P_{\ell_s}^t$, respectively, such that the following holds: there exist two bijections $\sigma:[s]\to \headset'(\net_1)$ and $\pi:[s]\to \tailset'(\net_k)$ such that for each $i\in[s]$, the first $t$ vertices of $P_i$ lie in~$\midd{X}_{\sigma(i)}$ and  the last $t$ vertices of $P_i$ lie in~$\midd{X}_{\pi(i)}$. 
    \end{description}
    Note that $\sum_{j=1}^m (|F_j|+1) = |F|$. 
    Hence, by setting $k=m$ in (b'), we get (b).

    So it remains to prove (b'). 
    In the base case $k=1$, we take $P_1,\dots,P_s$ to be $P_{1,1},\dots,P_{1,s}$ and $\sigma := \sigma_1, \pi := \pi_1$.
    For the inductive step, suppose $k \ge 2$ and we have found $P_1,\dots,P_s$ and $\sigma,\pi$ satisfying (b') for $k-1$.
    Our goal is to extend $P_1,\dots,P_s$ by $P_{k,1},\dots,P_{k,s}$ to obtain $P_1',\dots,P_s'$ satisfying (b') for $k$. We will also define appropriate $\sigma',\pi'$.
    To this end, we define an auxiliary bipartite graph $H$ with sides $\tailset'(\net_{k-1})$ and $\headset'(\net_k)$ such that $u\in\tailset'(\net_{k-1})$ is adjacent to $v\in \headset'(\net_k)$ in $H$ if there exists a blue $K_{t,t}$ with one part in $\down{X}_{u}$ and the other in $\up{X}_v$.
    \begin{claim}\label{claim:perfect-matching}
        $H$ has a perfect matching.
    \end{claim}
    \begin{proof}
        Suppose not.
        By Hall's marriage theorem, there exists a subset $S \subseteq \tailset'(\net_{k-1})$ with $|N_H(S)|<|S|$, where $N_H(S)$ is the neighborhood of $S$ in $H$.
        Take $T=\headset'(\net_k) \backslash N_H(S)$.
        Then there is no edge in $H$ between $S$ and $T$. Hence, for each $u \in S$ and each $v \in T$, there is no blue $K_{t,t}$ with one part in $\down{X}_{u}$ and the other in $\up{X}_v$.
        Observe also that for distinct $u,v\in \tailset'(\net_{k-1})$ or $u,v\in \headset'(\net_{k})$, there is no blue $K_{t,t}$ with one part in $X_{u}$ and the other in $X_v$. This follows from \cref{def:s-net}, since $\tailset'(\net_{k-1})$ and $\headset'(\net_{k})$ are paths from a root to a leaf in $F$, and so, for every two $u,v \in \tailset'(\net_{k-1})$ or $u,v\in \headset'(\net_{k})$, it holds that $u$ is a descendant of $v$ or vice versa. 
        Also, note that $|S|+|T|=|S|+s-|N_H(S)|\ge s+1$.
        So we see that the (at least $s+1$) vertices in $S \cup T$, along with the sets $(\down{X}_u)_{u\in S}$ and $(\up{X}_{v})_{v\in T}$, satisfy Item (a) in the lemma. This contradicts our assumption that (a) does not hold for $\net$.
    \end{proof}
    By the above claim, there exists a bijection $\tau:\tailset'(\net_{k-1})\to\headset'(\net_k)$ that specifies the perfect matching of $H$.
    Fix $i \in [s]$. 
    Due to (b') for $P_1,\dots,P_s$ (with $k-1$ in place of $k$), the last $t$ vertices of $P_i$ lie in $\midd{X}_{u}$ for $u=\pi(i) \in \tailset'(\net_{k-1})$ (and this is the only $u$ with this property). 
    Set $v := \tau(u)\in\headset'(\net_k)$.
    We know that the first $t$ vertices of $P_{k,j}$ lie in $\midd{X}_{v}$ for $j=\sigma_k^{-1}(v)$ (and this is the only $j$ with this property).
    We have $x_{k-1}<x_k$, and hence $V(F_{k-1}) < V(F_k)$, as $F$ is well-ordered. 
    Therefore, $u < v$, since $u \in V(F_{k-1})\cup\{x_{k-1}\}$ and $v \in V(F_k)\cup\{x_{k}\}$. 
    Now, by \cref{def:s-net}, we have $X_u < X_{v}$.
    As $(u, v) \in E(H)$, there is some blue $K_{t,t}$ with parts $L$ and $R$ such that $L\subseteq \down{X}_u$ and $R \subseteq \up{X}_{v}$.
    Also, both $X_u$ and $X_{v}$ are blue cliques. 
    Thus, we can extend $P_i$ by $P_i\to L\to R\to P_{k,j}$, giving a blue copy of $P_{\ell_i'}^t$ with $\ell_i' > \ell_i + \ell_{k,j}$. We denote this copy by $P'_i$.
    Doing the above for all $i \in [s]$, we obtain blue copies $P_1',\dots,P_s'$ of $P_{\ell_1'}^t,\dots,P_{\ell_s'}^t$, respectively,
    where
    \[
        \sum_{i=1}^s \ell_i' \geq
        \sum_{i=1}^s \ell_i + \sum_{i=1}^s \ell_{k,i}
        \ge \sum_{j=1}^{k-1} (|F_j|+1)r/3 + (|F_k|+1)r/3
        = \sum_{j=1}^{k} (|F_j|+1)r/3.
    \]
    As $\tau$ specifies a perfect matching in $H$, $P'_1,\dots,P'_s$ are vertex-disjoint. 
    Clearly, for every $1 \le i \le s$, the first $t$ vertices of $P_i'$ are the same as those of $P_i$, and thus lie in $\midd{X}_{\sigma(i)}$. The last $t$ vertices of $P_i'$ are the same as those of $P_{k,j}$ for $j=\sigma_k^{-1}(\tau(\pi(i)))$, and thus lie in $\midd{X}_{\pi_k(j)}$.
    Then $P'_1,\dots,P'_s$, $\sigma':=\sigma$ and $\pi':=\pi_k\circ\sigma_k^{-1}\circ\tau\circ\pi$ satisfy (b') for $k$. 
    This completes the inductive step, thus proving the \nolinebreak lemma.
\end{proof}

\begin{lemma} \label{cor:s-net is enough}
    Let $s \ge 1$, $n \ge t \ge 1$.
    Suppose $\net=(F,(X_v)_v)$ is an $s$-red-net of order $r\ge 3(4s^2)^t$ with $|F| \ge 3sn/r$.
    Then, $K_N$ contains a red $K_{s+1}$ or a blue $P_n^t$.
\end{lemma}
\begin{proof}
    We apply \cref{lem:s-net is enough} to $\net$, and consider the following two cases.
    \begin{description}
        \item[Case 1:] (a) holds for $\net$, i.e. there exist distinct subsets $A_1,A_2,\dots,A_{s+1}\subseteq [N]$, each of size $a=r/3$, such that for all $1 \le i < j \le s+1$, there is no blue $K_{t,t}$ with one part in $A_i$ and the other in $A_j$.
        Fix any $1 \le i < j \le s+1$.
        By the K\H{o}vari-S\'os-Tur\'an theorem~\cite{KST}, we have the following bound for the number of blue edges between $A_i$ and $A_j$:
        $$
        e_{\text{blue}}(A_i,A_j) 
        \leq z(a \times a, K_{t,t}) 
        < t^{1/t}a^{2-1/t}+ta
        \leq a^2/\binom{s+1}{2}.
        $$
        In the last inequality, we used the fact that $a = r/3 \ge (4s^2)^t \ge (2s^2)^tt$.
        Now, sample vertices $v_1\in A_1, \dots,v_{s+1}\in A_{s+1}$ independently and uniformly at random. 
        For every $1 \le i < j \le s+1$, the probability that $(v_i,v_j)$ is blue is smaller than $1/\binom{s+1}{2}$.
        Taking a union bound over all $i,j$, we get that with positive probability, $v_1,\dots,v_{s+1}$ form a red $K_{s+1}$.
        \item[Case 2:] (b) holds for $\net$, i.e. there exist $P_1,\dots,P_s$ such that for each $i \in [s]$, $P_i$ is a blue copy of $P_{\ell_i}^t$ for some $\ell_i \ge t$, and $\sum_{i=1}^s \ell_i \ge |F|r/3$.
        By averaging, there is $i\in[s]$ such that $\ell_{i} \ge |F|r/(3s) \ge n$. 
        Thus, $P_{i}$ contains a blue copy of $P_n^t$.
    \end{description}    \vspace{-0.5cm}\end{proof}

\subsection{Putting it all together}
\begin{proof}[Proof of \Cref{thm:clique vs power-path}]
Set $r = 3(4s^2)^t$, $M = 2^{s-1}n$, and $N = M\cdot R(K_{s+1},K_{sr})$. 
Note that $R(K_{s+1},K_m) \leq \binom{m+s-1}{s} \leq m^s$ (by the Erd\H{o}s-Szekeres bound~\cite{ES}). Hence,
$
N \leq 2^{s-1} n\cdot (rs)^s  \leq (24s^3)^{st} n.
$
Fix a red/blue edge-coloring of $K_N$ and suppose by contradiction that there is no red~$K_{s+1}$ and no blue~$P_n^t$. Then every $R(K_{s+1},K_{sr})$ vertices contain a blue clique of size $sr$. 
As $N = M \cdot R(K_{s+1},K_{sr})$, this means that we can find blue cliques $V_1 < \dots < V_M$ of size $sr$ each. Partition $V_i = X_i^{(0)} \cup \dots \cup X_i^{(s-1)}$ with $|X_i^{(j)}| = r$ for all $0 \leq j \leq s-1$, and $X_i^{(s-1)} < X_i^{(s-2)} < \dots < X_i^{(0)}$. 

For each $1 \leq i \leq M$ and $0 \leq j \leq s-2$, let $\chi_j(i)$ be the maximum $\ell$ such that there exists a blue copy of $P_\ell^t$ whose last $t$ vertices belong to $X_i^{(j+1)} \cup \dots \cup X_i^{(s-1)}$. 
By definition, for every $1 \le i \le M$,
\begin{equation}\label{eq:chi_functions monotone}
\chi_0(i) \geq \chi_1(i) \geq \dots \geq \chi_{s-2}(i) \geq t.
\end{equation}
Also, $\chi_j(i) < n$ for all $j$ because there is no blue $P_n^t$. By \cref{lem:chi-functions} (with $q = s-1$), there is a $(\chi_0,\dots,\chi_{s-2})$-forest $F$ with $|L(F)| \geq M/2^{s-2} - n \geq n$. The following is the key property we need:
    \begin{claim}\label{claim:blue K_{t,t}}
    	Let $0 \leq d < d' \leq s-1$, let $a \in F$ at depth $d$ and let $a'$ be a descendant of $a$ at depth $d'$. Then there is no blue $K_{t,t}$ with one part in $X_a^{(d)}$ and the other part in $X_{a'}^{(d')}$.  
    \end{claim}
    \begin{proof}
	Let $a=b_0,b_1,\dots,b_{d'-d}=a'$ be the unique path from $a$ to $a'$ in $F$. Then $b_i$ is at depth $d+i$. 
        We have $\chi_{d+i}(b_i) \geq \chi_{d+i}(b_{i+1})$ for every $0 \leq i < d'-d$, by the definition of a $(\chi_0,\dots,\chi_{s-2})$-forest (\cref{def:chi values}). Also, $\chi_{d+i}(b_{i+1}) \geq \chi_{d+i+1}(b_{i+1})$ for every 
        $0 \leq i \leq d'-d-2$,
        by \eqref{eq:chi_functions monotone}. 
        So $\chi_{d+i}(b_i) \geq \chi_{d+i+1}(b_{i+1})$ for $0 \leq i \leq d'-d-2$, meaning that the sequence $(\chi_{d+i}(b_i))_{i=0}^{d'-d-1}$ is non-increasing. 
        Setting $i=0$ and $i=d'-d-1$, we get $\chi_d(a) = \chi_d(b_0) \geq \chi_{d'-1}(b_{d'-d-1}) \geq \chi_{d'-1}(b_{d'-d}) =  \chi_{d'-1}(a')$, where the last inequality again uses the definition of a $(\chi_0,\dots,\chi_{s-2})$-forest.  
	
        Now suppose by contradiction that there is a blue $K_{t,t}$ with sides $A \subseteq X_a^{(d)}$ and $A' \subseteq X_{a'}^{(d')}$. 
        By the definition of $\chi_d(\cdot)$, there exists a blue copy $P$ of $P_{\chi_d(a)}^t$ whose last $t$ vertices belong to $X_a^{(d+1)} \cup \dots \cup X_a^{(s-1)}$.
        We can extend $P$ by adding the sets $A$ and $A'$, using that both $A$ and $A'$ induce blue cliques, and all edges between $A$ and the last $t$ vertices of $P$ are blue (because $V_a,V_{a'}$ are blue cliques). 
        Extending $P$ in this way gives a blue copy $Q$ of $P_{\chi_d(a)+2t}^t$, whose last $t$ vertices are $A' \subseteq X_{a'}^{(d')}$. Hence, $\chi_{d'-1}(a') \geq \chi_d(a)+2t >\chi_d(a)$, a contradiction to $\chi_d(a) \geq \chi_{d'-1}(a')$. 
    \end{proof}

    We now use \cref{claim:blue K_{t,t}} to find an $s$-red-net. 
    For every $0 \leq d \leq s-1$ and every $v \in V(F)$ at depth $d$ in $F$, define $X_v := X_v^{(d)} \subseteq V_v$. 
    Then $X_v$ is a blue clique of size $r$. For every $v,u \in V(F)$ with $v < u$, we have $V_v < V_u$ and hence $X_v < X_u$. 
    Also, by \cref{claim:blue K_{t,t}}, for every $v \in V(F)$ and every descendant $u$ of $v$, there is no blue $K_{t,t}$ with one part in $X_v$ and one part in $X_u$. 
    Hence, $(F,(X_v)_v)$ is an $s$-red-net of order $r$. 
    Also, $|F| \geq n \geq 3sn/r$. 
    By \cref{cor:s-net is enough}, there is a red $K_{s+1}$ or a blue $P_n^t$. This completes the proof. 
\end{proof}

\section{Concluding remarks and open problems}\label{sec:concluding remarks}
In \Cref{thm:clique vs power-path}, we showed that 
$R_{<}(K_{s+1},P_n^t) \leq s^{O(st)} n$. As for lower bounds, it holds that $R_{<}(K_{s+1},P_n^t) > (R(K_{s+1},K_{t+1}) - 1) \cdot (n-1)/t$. Indeed, partition the vertices into $(n-1)/t$ intervals of size $R(K_{s+1},K_{t+1}) - 1$, and on each of the intervals, put a red/blue coloring with no red $K_{s+1}$ and no blue $K_{t+1}$. All edges between the intervals are blue. Then there is no red $K_{s+1}$ and no blue $P_n^t$, because a blue $P_n^t$ would have to contain $t+1$ vertices from one of the intervals which appear consecutively in the $P_n^t$, and hence must form a blue $K_{t+1}$. 
Combining our upper and lower bounds, we obtain $s^{\Omega(t)} \cdot n \leq R_{<}(K_{s+1},P_n^t) \leq s^{O(st)} \cdot n$ (when say $t \ll s$). It would be interesting to determine the correct dependence of the exponent on $s$ and $t$. 

In \Cref{thm:Kn Pn^t} we proved that $R_<(P_n^t,K_n) = O_t(n^{t(2t-1)})$. It would be interesting to improve the exponent further.
    \begin{conjecture}
        $R_<(P_n^t,K_n) \leq n^{O(t)}$.
    \end{conjecture}
This conjecture, if true, would be tight (up to the implied constant in the exponent) because $R_<(P_n^t,K_n) \geq R(K_{t+1},K_n) \geq \tilde{\Omega}(n^{(t+2)/2})$, see~\cite{BK,Spencer}. 
We note that our proof method for \Cref{thm:Kn Pn^t} can be used to prove $R_<(P_n[t],K_n) \leq n^{O(t)}$, where $P_n[t]$ is the $t$-blowup of the monotone path $P_n$ (i.e., $P_n[t]$ is obtained by replacing each vertex of $P_n$ with $t$ vertices and replacing edges with complete bipartite graphs). This bound is tight because $R_<(P_n[t],K_n) \geq R(K_{t,t},K_n) \geq n^{\Omega(t)}$. 
\begin{proposition}
$R_<(P_n[t],K_n) \leq (2tn^3)^{2t-1}$.
\end{proposition}
\begin{proof}[Proof sketch]
The proof is similar to that of \Cref{thm:Kn Pn^t}.
A {\em semi-red $t$-clique} is a clique $x_1 < \dots < x_t$ such that all edges $x_1x_i$ ($2 \leq i \leq t$) and $x_ix_t$ ($1 \leq i \leq t-1$) are red. 
It can be shown that every set of $N := (t-2)n^2$ vertices contains a blue $K_n$ or semi-red $t$-clique. Indeed, if there is no vertex with forward red degree at least $(t-2)n$, then one can greedily find a blue clique of size $\frac{N}{(t-2)n} \geq n$. Thus, there is a vertex $x_1$ with forward red degree at least $(t-2)n$. By the same argument inside the forward red neighbourhood of $x_1$, we can find a vertex $x_t$ with at least $t-2$ backward red neighbours $x_2,\dots,x_{t-1}$. Now $x_1 < x_2 < \dots < x_t$ is a semi-red $t$-clique. 

One can then show, using the argument from \Cref{claim:clique_chain}, that every set of $((t-2)n^2 - 1)n + 1$ vertices contains a blue $K_n$ or a chain of $n$ semi-red $t$-cliques. 
Then, we essentially repeat the proof of \Cref{thm:Kn Pn^t}: consider a red/blue coloring of $K_N$ with $N=(2tn^3)^{2t-1}$.
We assume that there is no red $P_n[t]$ or blue $K_n$.
For a semi-red clique $x_1<\dots<x_t$, define $\chi(x_1,\dots,x_t)$ to be the largest number $\ell$ such that $x_1,\dots,x_t$ are the last $t$ vertices of a red $P_{\ell}[t]$.
Then $1 \leq \chi(x_1,\dots,x_t) \leq n-1$.
Using the choice of $N$, one can obtain vertices $x_1 < \dots < x_{t-1} < z_1 < \dots < z_{t-1}$, a set $Y$ with $|Y| \geq tn^2$, and a value $1 \leq c \leq n-1$, such that for every $y \in Y$, $\{x_1,\dots,x_{t-1},y\},\{y,z_1,\dots,z_{t-1}\}$ are semi-red, $\chi(x_1,\dots,x_{t-1},y) = c$, and $\chi(y,z_1,\dots,z_{t-1}) \leq c$. 
Then, by finding a semi-red $(t+2)$-clique $y_1 < \dots < y_{t+2}$ inside $Y$, we can extend a longest $P_c[t]$ ending at $x_1,\dots,x_{t-1},y_1$ by adding the vertices $y_2,\dots,y_{t+1},y_{t+2},z_1,\dots,z_{t-1}$. 
This gives $\chi(y_{t+2},z_1,\dots,z_{t-1}) > \chi(x_1,\dots,\chi_{t-1},y_1)$, a contradiction. 
\end{proof}

In \Cref{thm:P^t} we obtain a new bound on $R_<(P_n^t,P_n^t)$, whose exponent grows linearly with $t$. 
Since we do not have a corresponding lower bound, we wonder whether the following might be true.
    \begin{problem}
    Is there a constant $C$ independent of $t$, such that $R_<(P_n^t,P_n^t) = O_t(n^C)$?
    \end{problem}
\noindent
Even improving the exponent to $o(t)$ would be interesting. 

In \Cref{thm:non-increasing set} we proved that $g(n,s) = O_s(n^C)$ for a constant $C = C(s)$. 
It may be interesting to determine the order of growth of (the optimal such) $C(s)$. Our proof of \Cref{thm:non-increasing set} gives an upper bound on $C(s)$ of the order~$s^s$, and this is likely far from optimal. Does $C(s)$ grow polynomially in~$s$? Also, for $s=3$, is it true that $g(n,3) = O(n)$? 

One could also consider the analogous extremal functions for weaker notions of a non-increasing triple (see \Cref{def:non-increasing}). There are two such notions: one is to require that $\chi(x,y) \geq \chi(x,z) \geq \chi(y,z)$, and one is only to require that $\chi(x,y) \geq \chi(y,z)$. For each of these notions, how large should~$N$ be to guarantee a non-increasing set of size $s$ in an $n$-coloring of $K_N$?

\bibliographystyle{plain} 

\end{document}